\documentclass[12pt]{article}
\usepackage{amsmath}
\usepackage{graphicx}
\usepackage{enumerate}
\usepackage[english]{babel}
\usepackage[authoryear,round]{natbib}
\usepackage{url} 

\usepackage{amsthm,amsfonts,amssymb, bm}
\usepackage{booktabs, float, comment}
\usepackage{multirow}
\usepackage[linesnumbered,ruled,vlined]{algorithm2e}
\usepackage{bbm}
\usepackage{setspace}

\DeclareMathOperator*{\argmax}{argmax}
\DeclareMathOperator*{\argmin}{argmin} 
\def\0{\mathbf{0}}
\def\cD{\mathcal{D}}
\def\Dbb{\mathbb{D}}

\def\Ebb{\mathbb{E}}
\def\hG{\hat{G}}
\def\cG{\mathcal{G}}
\def\cH{\mathcal{H}}
\def\bI{\mathbf{I}}
\def\cL{\mathcal{L}}
\def\Lbb{\mathbb{L}}
\def\real{\mathbb{R}}

\def\cS{\mathcal{S}}
\def\cW{\mathcal{W}}
\def\vx{\bm{x}}
\def\cX{\mathcal{X}}
\def\cY{\mathcal{Y}}

\def\vtheta{{\bm{\theta}}}
\def\hvtheta{\hat{\vtheta}}
\def\vphi{{\bm{\phi}}}
\def\hvphi{\hat{\vphi}}

\def\Var{\text{Var}}
\def\KL{\text{KL}}
\newcommand{\indep}{\;\, \rule[0em]{.03em}{.67em} \hspace{-.25em}
	\rule[0em]{.65em}{.03em} \hspace{-.25em}
	\rule[0em]{.03em}{.67em}\;\,}

\newtheorem{theorem}{Theorem}[section]
\newtheorem{lemma}{Lemma}[section]
\newtheorem{lemma*}{Lemma}

\newtheorem{corollary}{Corollary}[section]

\newcommand{\blind}{1}

\usepackage{multirow}
\usepackage{array}
\newcommand{\PreserveBackslash}[1]{\let\temp=\\#1\let\\=\temp}
\newcolumntype{C}[1]{>{\PreserveBackslash\centering}p{#1}}
\usepackage{colortbl}
\usepackage[T1]{fontenc}
\usepackage[utf8]{inputenc}
\usepackage[table]{xcolor}
\usepackage{xcolor}

\newcommand{\blue}[1]{{\color{blue}#1}}

\newcommand{\red}[1]{{\color{red}#1}}

\colorlet{tableheadcolor}{gray!25}
\newcommand{\headcol}{\rowcolor{tableheadcolor}} %
\colorlet{tablerowcolor}{gray!10}
\newcommand{\topline}{\arrayrulecolor{black}\specialrule{0.1em}{\abovetopsep}{0pt}%
	\arrayrulecolor{tableheadcolor}\specialrule{\belowrulesep}{0pt}{0pt}%
	\arrayrulecolor{black}}
\newcommand{\midline}{\arrayrulecolor{tableheadcolor}\specialrule{\aboverulesep}{0pt}{0pt}%
	\arrayrulecolor{black}\specialrule{\lightrulewidth}{0pt}{0pt}%
	\arrayrulecolor{white}\specialrule{\belowrulesep}{0pt}{0pt}%
	\arrayrulecolor{black}}

\newcommand{\bottomline}{\arrayrulecolor{white}\specialrule{\aboverulesep}{0pt}{0pt}%
	\arrayrulecolor{black}\specialrule{\heavyrulewidth}{0pt}{\belowbottomsep}}%

\begin{document}

\def\spacingset#1{\renewcommand{\baselinestretch}%
{#1}\small\normalsize} \spacingset{1}

\if1\blind
{
\title{A Deep Generative Approach to Conditional Sampling}
\author{Xingyu Zhou\thanks{Xingyu Zhou and Yuling Jiao constributed equally to this work.} \thanks{Department of Statistics and Actuarial Science, University of Iowa, Iowa City, Iowa 52242, USA. Email: xingyu-zhou@uiowa.edu}
\ Yuling Jiao$^*$\thanks{School of Mathematics and Statistics, Wuhan University, Wuhan, Hubei, China 430072.  Email: yulingjiaomath@whu.edu.cn}
\ Jin Liu\thanks{Centre for Quantitative Medicine,
Duke-NUS Medical School, Singapore. Email: jin.liu@duke-nus.edu.sg} \ and
 Jian Huang\thanks{Corresponding author. Department of Statistics and Actuarial Science, University of Iowa, Iowa City, Iowa 52242, USA. Email: jian-huang@uiowa.edu}}
\date{\today}

  \maketitle
} \fi

\if0\blind
{
  \bigskip
  \bigskip
  \bigskip
  \begin{center}
    {\Large \bf A Deep Generative Approach to Conditional Sampling}
\end{center}
  \medskip
} \fi


\begin{abstract}
We propose a deep generative approach to sampling from a conditional distribution based on a unified formulation of conditional distribution and generalized nonparametric regression function using the noise-outsourcing lemma. The proposed approach aims at learning a conditional generator so that a random sample from the target conditional distribution can be obtained by the action of the  conditional generator on a sample drawn  from a reference distribution. The conditional generator is estimated nonparametrically with neural networks by matching appropriate joint distributions using Kullback-Liebler divergence. An appealing aspect of our method is that it allows either of or both the predictor and the response to be high-dimensional and can handle both continuous and discrete type predictors and responses. We show that the proposed method is consistent in the sense that the conditional generator converges in distribution to the underlying conditional distribution under mild conditions. Our numerical experiments with simulated and benchmark image data validate the proposed method and demonstrate that it outperforms several existing conditional density estimation methods.
\end{abstract}


\vfill

\newpage
\spacingset{1.5} 
\section{Introduction}

A fundamental problem in statistics and machine learning is how to model the relationship between a response  $Y$ and a predictor $X$. Such a model can be used for predicting the values of $Y$ based on the new observations of $X$ and for assessing the variation in $Y$ for a given value of $X$.
Regression models that focus on estimating the conditional mean or median of the response given the predictor have been widely used for such purposes. However, in problems when the conditional distribution is multimodal or asymmetrical, conditional mean and median are no longer adequate for modeling the relationship between $Y$ and $X$. In general, to completely understand how the response depends on the predictor, it becomes necessary to learn the conditional distribution, which provides a full description of the relationship between the response variable and the predictor.
Conditional distribution also plays a central role in many important areas, including
representation learning \citep{bengio2013representation},
sufficient dimension reduction \citep{li1991sliced, cook1998regression}, graphical models \citep{bishop2006},  nonlinear independent component analysis \citep{hyva1999},  among others.

In this paper, we propose a nonparametric generative approach to sampling from a conditional distribution. For convenience, we shall refer to the proposed method as the generative conditional distribution sampler (GCDS). For a given value of the predictor  $X=x$, GCDS aims at estimating a function $G(\eta, x)$ of $\eta$ and $x$,   where $\eta$ is a random variable from a simple reference distribution such as normal or uniform,  such that  $G(\eta, x)$ follows the conditional distribution of $Y$ given $X=x$.  Such a function is called a conditional generator. To sample from the conditional distribution, we only need to calculate $G(\eta, x)$ after generating $\eta$ from the reference distribution. Therefore, the conditional generator $G$ contains all the information about the conditional distribution of $Y$ given $X$.
We estimate the conditional generator $G$ nonparametrically using neural networks.
An appealing feature of GCDS is that it is applicable to the settings when either of or both $X$ and $Y$ are high-dimensional such as in the problems of image data analysis.

There is an extensive literature on nonparametric conditional density estimation. The prevailing approaches are based on smoothing methods, including kernel smoothing and local polynomials
\citep{rosenblatt1969, scott1992,
hyndman1996estimating, chen2001,
hall2005, bott2017}. Typically, the joint density of $(X, Y)$ and the marginal density of $X$ are first estimated using unconditional kernel density estimators. Then, the conditional density estimator is obtained as the ratio of the estimated joint density over the estimated marginal density.
Another approach is to transform the problem of estimating a conditional density to a suitably formulated  regression problem \citep{fan1996, fan2004} and use the method for nonparametric regression for conditional density estimation.
Nearest neighbors method has also been used in estimating conditional density and conditional quantiles through kernel smoothing
\citep{
zhao1985, bhattacharya1990}.
Approaches based on  expanding the conditional density function in terms of certain basis functions have also been developed  
 \citep{sugi2010, izbicki2016nonparametric}.
The method proposed by \citet{izbicki2017converting} approximates the conditional density using orthogonal basis functions and transform the problem of conditional density estimation into a regression problem. A common feature of these methods is that they seek to estimate the functional form of the conditional density.

However, these existing conditional density estimation
methods 
do not work well for problems with high-dimensional data.
In particular, they suffer from the so-called ``curse of dimensionality'', that is, their performance deteriorates dramatically as the dimensionality of the dependent variable or the response variable becomes relatively large. Indeed, most conditional density estimators can only effectively handle up to a few covariates. 
In addition, most of these methods focus on the case when the response $Y$ is a scalar variable and cannot handle the case of a high-dimensional response vector.

The proposed approach is inspired by the recently developed {generative adversarial networks}
 (GAN) \citep{goodfellow14}.
Instead of estimating the functional form of the conditional density, GCDS is a generative learning approach that seeks to estimate a conditional sampler.
The basis of GCDS is a unified formulation of the conditional density estimation and the generalized nonparametric regression based on the noise-outsourcing lemma in probability theory \citep{kall2002, austin2015}. By this lemma, the problem of nonparametric conditional density estimation is equivalent to a generalized nonparametric regression problem. This equivalency implies that, for any given $X=x$,  we can estimate the conditional generator $G(\eta, x)$ so that a random sample from the conditional distribution can be obtained based on this function using a random sample $\eta$ from a reference distribution, such as the uniform or the standard normal distribution. The estimation of the conditional generator $G$ is achieved through matching appropriate joint distributions using the Kullback-Liebler divergence and its variational form.
We take advantage of the abilities of neural networks in approximating high-dimensional functions and estimate $G$ nonparametrically using deep neural networks.

There are several advantages of GCDS over the classical methods for conditional density estimation. First, there is no restriction on the dimensionality of the response variable, while the classical methods typically only consider the case of a scalar response variable. Indeed, our methods allow either of or both  the predictor and the response to be high-dimensional. Second, GCDS can handle both continuous and discrete type predictors and responses. 
Third, since our method learns a generative function for the underlying conditional distribution based on a simple reference distribution, it is easy to obtain estimates of the summary measures of the underlying conditional distribution, including the conditional moments and quantiles by Monte Carlo.
In comparison, it is cumbersome to do so based on the traditional conditional density estimation methods, since it involves numerical integrations that are difficult to implement in high-dimensional settings. Fourth, we demonstrate that the proposed method works for complex and high-dimensional data problems such as image generation and reconstruction. The traditional conditional density estimation methods are not able to deal with such problems. Finally, we show that  GCDS is consistent in the sense that the samples it generates converge weakly to the underlying target conditional distribution. To the best of our knowledge, such a result is the first of its kind in the context of deep generative learning.

In the remainder of this paper, we first describe a generative representation of conditional distribution based on the noise-outsourcing lemma, and explain that sampling from a conditional distribution can be achieved by using a conditional generator. This provides the theoretical foundation for the
distribution matching method proposed in Section 3. The distribution matching is carried out by using the variational form of the $f$-divergence, which includes the Kullback-Liebler divergence as a special case. In Section 4 we establish the consistency of the conditional generator in the sense that the joint distribution of $X$ and generated sample converges to the joint distribution of $(X,Y)$. In Section 5 we conduct extensive simulation studies to evaluate the finite sample performance
of the proposed method and illustrate its application to an image generation and reconstruction problems using  benchmark image data. Concluding remarks are given in Section 6. Additional numerical experiment results and technical details are provided in the supplementary material.

\section{Generative representation of conditional distribution}

Consider a pair of random vectors  $(X, Y) \in \cX \times \cY$,  where $X$ is a vector of predictors and $Y$ is a vector of response variables or labels. For regression problems, we have $\cY \subseteq \real^q$ with $q \ge 1$;  for classification problems, $\cY$ is a set of finite many labels. We assume the space $\cX \subseteq \real^d$ with $d \ge 1$. The predictor $X$ can contain both continuous and categorical components.
Suppose $(X, Y) \sim P_{X, Y}$ with marginal distributions
$X \sim P_X$ and $Y\sim P_Y$. Denote the conditional distribution of $Y$ given $X$ by $P_{Y|X}$. For a given value $x$ of $X$, we also write the conditional distribution as $P_{Y|X=x}$.
Let $\eta$ be a random vector independent of $X$ with a known distribution $P_{\eta}$.
For example, we can take $P_{\eta}$ to be the standard  multivariate normal $N(\0,
\bI_m),$ for a given $m \ge 1.$ We note that $m$ does not need to be the same as $q$, the dimension of $Y.$

Our goal is to find a function  $G: \real^m \times \cX \mapsto \cY$ such that the conditional distribution of $G(\eta,X)$ given $X=x$
is the same as the conditional distribution of $Y$ given $X=x$.
Since $\eta$ is independent of $X$, it is equivalent to finding a $G$ such that
\begin{equation}
\label{dm1a}
G(\eta, x) \sim P_{Y|X=x},\ x \in \cX,
\end{equation}
Thus to sample from the conditional distribution $P_{Y|X=x}$, we can first sample an $\eta\sim P_{\eta}$, then calculate $G(\eta, x)$. The resulting value $G(\eta, x)$  is a sample from
$ P_{Y|X=x}.$

Does such a function $G$ exist? The existence of $G$ is guaranteed by the noise-outsourcing lemma from probability theory under minimal conditions (Theorem 5.10 in
\cite{kall2002},  Lemma 3.1 in \cite{austin2015}).

\begin{lemma}
\label{NOLemma}
(Noise-outsourcing lemma). Let  $(X, Y)$ be a random pair taking values in $\cX\times \cY$ with joint distribution $P_{X,Y}$.
Suppose $\cY$ is a standard Borel space.
Then there exist a random vector $\eta \sim P_{\eta}=N(\0, \bI_m)$ for any given $m \ge 1$
and a Borel-measurable function $G : \real^m \times \cX \to \cY$ such that $\eta$ is independent of $X$ and
\begin{equation}
\label{NOa}
(X, Y)= (X, G(\eta, X))  \ \text{almost surely.}
\end{equation}
\end{lemma}
Because $\eta$ and $X$ are independent,
any $G$ satisfying \eqref{NOa}  also satisfies \eqref{dm1a}, that is, $G(\cdot, x)$ is a conditional generator for  $P_{Y|X=x}.$
%
In the original noise-outsourcing lemma, the distribution $P_{\eta}$ is a
uniform distribution on $[0, 1]$. In Lemma \ref{NOLemma}, $P_{\eta}$ is taken to  be
$N(\0, \bI_m)$ with $m \ge 1$. This is more convenient in applying GCDS
 when it is better to use a random vector in stead of a random scalar as the noise source.
In the supplementary material, we show that we can indeed take $P_{\eta}$ to be
 $N(\0, \bI_m)$
based on the original noise-outsourcing lemma.

Lemma \ref{NOLemma} provides a unified view of
conditional distribution estimation and (generalized) nonparametric regression. To see this, it is informative to reverse the order of \eqref{dm1a} and write it as
\begin{equation}
\label{dm1b}
Y|X=x \sim G(\eta, x), \ x \in \cX.
\end{equation}
This expression shows that the problem of finding $G$ is similar to that of estimating a generalized  regression function nonparametrically by matching the conditional distributions. Therefore, $G$ can also be considered a generalized regression function.
The standard formulation of nonparametric regression with an additive error is a special case of \eqref{dm1b}.
Indeed, if we assume $G(\eta, x)=G_0(x)+\eta$ with $\Ebb(\eta|X)=0$, then \eqref{dm1b}
leads to the standard nonparametric regression model
$\Ebb(Y|X=x) = G_0(x)$.

Sampling from a conditional distribution generally cannot be done by simply using the existing methods for sampling from an unconditional distribution. This can be illustrated as follows. For any given $x$, the problem is to find a function $G_x(\eta)$ such that
$
G_x(\eta) \sim P_{Y|X=x},
$
where we use $x$ as the subscript of $G$ to indicate that the form of $G$ depends on $x$.
If $X$ is a discrete random variable and only takes the values in a finite set, then we can simply find the function $G_x$ for each $x$ using the existing generative methods such as  GAN  \citep{goodfellow14}.
However, this is not feasible if $X$ is a continuous-type random variable.
In general, the methods for generating samples from an unconditional distribution cannot be directly applied to find a function of $\eta$ for generating samples from the conditional distribution $P_{Y|X}$.

To get around this difficulty, we note that matching the conditional  distribution of $G(\eta, x)$  with $P_{Y|x}$ for a given $x \in \cX$ is equivalent to matching the joint distribution of $(X, G(\eta, X))$ and the joint distribution of $(X, Y)$,  if the same marginal distribution of $X$ is involved. This can be easily seen as follows.
Let $T=G(\eta, X)$. Then $P_{T|X}=P_{Y|X}$ if and only if
$ P_{T|X}P_X=P_{Y|X}P_X$ on the support of $(X, Y)$, that is, $P_{X, T}=P_{X, Y}$.
We summarize this simple but key observation in the following lemma.

\begin{lemma}
\label{distm}
Suppose that $\eta$ is independent of $X$.  Then $G(\eta, x) \sim P_{Y|X=x},\ x \in \cX$ if and only if
\begin{equation}
\label{dm2}
(X, G(\eta, X)) \sim (X, Y).
\end{equation}
\end{lemma}
Because of (\ref{dm2}), we refer to $G$ as a conditional generator, since given $X=x$, $G(\eta, x) \sim P_{Y|X=x}.$
Lemma \ref{distm} shows that finding a $G$ such that (\ref{dm1a}) holds amounts to finding a $G$ such that the joint distribution of $(X, G(\eta, X))$ is the same as that of $(X, Y).$

It is clear that the conditional generator satisfying \eqref{dm2} contains all the information about the conditional distribution of $Y$ given $X$.
For example, consider
the conditional expectation $g(x)=\Ebb(Y|X=x)$ and the conditional variance
$v(x)=\Var(Y|X=x)$.
By  \eqref{dm2}, we have
$g(x)= \Ebb_{\eta\sim P_{\eta}}G(\eta,x)$  and
$v(x) = \Var_{\eta}[G(x,\eta)].$
Therefore, we can calculate $g$ and $v$ based on $G$.
Although it is difficult to calculate these functions exactly, it is easy to approximate them via Monte Carlo. Specifically, let $\eta_1, \ldots, \eta_J$ be a random sample generated from $P_{\eta}$, then we can approximate $g(x)$ and $v(x)$ by
\begin{equation}
\label{mont}
\tilde{g}(x) = \frac{1}{J}\sum_{j=1}^J G(\eta_j, x)
\ \text{ and } \
\tilde{v}(x) = \frac{1}{J} \sum_{j=1}^J [G(\eta_j,x)-\tilde{g}(x)]^2.
\end{equation}
Since it is easy and inexpensive to generate random samples from
$P_{\eta}$,
for any given $x$ we can easily accurately approximate the summary measures such as moments and quantiles of the conditional distribution $P_{Y|X=x}$ based on $\{G(\eta_j, x), j=1, \ldots, J \}$ for a sufficiently large $J$.

\section{Distribution matching estimation}

\subsection{Adversarial generative networks}
The generative adversarial networks  (GAN) \citep{goodfellow14}
is an approach to learning a high-dimensional (unconditional) distribution. It is formulated as a minimax adversarial game between two players, a generator $G$ and a discriminator $D$. The discriminator $D$ is parameterized using a neural network that serves as a witness to distinguish between a sample $Y$ from the data distribution and a sample from the generative model. The generator $G(\eta)$ maps samples $\eta$ from the reference distribution $P_{\eta}$ to the data distribution.
The generator $G$ is trained to maximally confuse the discriminator into believing that samples it generates come from the data distribution.
Formally, GAN solves the minimax optimization problem:
\begin{equation}
\label{gan1}
\min_G \max_D \Ebb_{Y\sim P_{\text{data}}}\log D(Y)
+\Ebb_{\eta \sim P_{\eta}}\log[1-D(G(\eta))].
\end{equation}

Conditional generative adversarial networks (cGAN) \citep{mirza2014cgan} estimate the distribution of the images conditioning on some auxiliary information, especially class labels. Similar to GAN, it solves a two-player minimax game
using an objective function with the same form as (\ref{gan1}).
See equation (2) in \citet{mirza2014cgan}.
cGAN performs the conditioning by feeding the class label information into the neural networks for the discriminator and the generator as additional input layer.
However, cGAN does not work well for data generation with continuous conditions.

\subsection{$f$-divergence and its variational form}

While the minimax formulation has an attractive intuitive interpretation as a two-player game between the generator and the discriminator, it is helpful to understand it as the dual form of the primal problem of minimizing the Jensen-Shannon divergence between the data distribution and the distribution of the generator \citep{goodfellow14}. By considering a general discrepancy measure between two distributions such as the $f$-divergence, one can formulate a class of generative learning methods including GAN as a special case \citep{nowozin2016f}.

By Lemma \ref{distm}, we can estimate $G$ by matching the distribution of $(X, G(\eta, X))$ with the distribution of $(X, Y).$ For this purpose, we first describe the $f$-divergence and its variational form.
Let $P$ and $Q$ be two probability distributions on $\mathbb{R}^{d}$.
Let $p$ and $q$ be the density functions of  $P$ and $Q$  with respect to a common dominant 
measure, respectively.
Suppose $Q$ is absolutely continuous with respect to $P.$
The $f$-divergence \citep{ali1966general} of $Q$ with respect to $P$ is defined by
\begin{equation}\label{fdiv}
\mathbb{D}_f(q\Vert p) = \int f\left(\frac{q(z)}{p(z)}\right) p(z) dz,
\end{equation}
where $f: \real^+ \rightarrow \real $ is a
convex function with $f(1) = 0$ and is strictly convex at $x=1$.
A basic property of the $f$-divergence following from Jensen's inequality is that
$\Dbb_f(q\Vert p) \ge 0$ for every $q, p$ and $\Dbb_f(q\Vert p) = 0$ if and only if $q=p$.

The Kullbak-Liebler (KL) divergence is an important special case with $f(x)=x\log x$, which has a simple expression
\begin{equation}
\label{KLc}
\Dbb_{\text{KL}}(q \Vert p) =\int \frac{q(z)}{p(z)}\log \left(\frac{q(z)}{p(z)}\right) p(z) dz
= \int \log \left(\frac{q(z)}{p(z)}\right) q(z) dz.
\end{equation}
%
Let $r = {q}/{p}$ be the density ratio of the densities $q$ and $p$. It is convenient to express the KL divergence as
\[
\Dbb_{\text{KL}}(q \Vert p) =\int \log \left(\frac{q(z)}{p(z)}\right) q(z) dz=
\Ebb_{Z\sim q} [\log r(Z)].
\]

A useful representation of the $f$-divergence is its variational form. We will use it to construct an objective function for training the conditional generator $G$. The variational form is based on  the Fenchel conjugate of $f$ \citep{rockafellar1970convex},
defined as
\(
f^*(t) = \sup_{x \in \real}\{ t x - f(x)\}, t\in \real.
\)
Then the $f$-divergence has the following variational formulation \citep{keziou2003dual,nguyen2010estimating}. 
\begin{lemma}\label{lem2}
Let $\cD$ be a class of measurable functions $D: \real^d \to \real$.
Suppose $f$ is a differentiable convex function. Then
\begin{equation}\label{fdual}
	\mathbb{D}_f(q\Vert p) \ge  \sup_{D \in \cD}
[\mathbb{E}_{Z\sim q} D(Z)-\mathbb{E}_{W\sim p} f^*(D(W))],
	\end{equation}
where the equality holds if and only if
$ f^{\prime}(q/p) \in \cD$ and the supremum is attained at
$D^* = f^{\prime}(q/p).$
\end{lemma}

Commonly used divergence measures,  including the KL divergence, the Jensen-Shannon (JS) divergence and the $\chi^2$-divergence, can be considered special cases of $f$-divergence. We give a proof of Lemma \ref{lem2} and the expressions of the conjugate functions and variational forms of these divergence measures in the supplementary material.

\subsection{Distribution matching estimation via $f$-divergence}
We now apply Lemmas \ref{distm} and \ref{lem2} to construct the objective function for estimating the conditional generator $G$.
Let $p_{X, G(\eta, X)}$ and $p_{X, Y}$ be the densities of
$(X, G(\eta, X))$ and $(X, Y)$, respectively. At the population level, we seek to find a conditional generator $G^*$ that minimizes the $f$-divergence
$\Dbb_f (p_{X, G(\eta, X)}\Vert  p_{X, Y}).$

\begin{lemma}
\label{KLM}
A function $G^*: \real^m \times \cX \to \cY$ is a minimizer of the $f$-divergence $\Dbb_{f} (p_{X, G(\eta, X)}\Vert p_{X, Y})$,
\begin{equation}
\label{KLa}
G^* \in  \argmin_{G}  \Dbb_{f} (p_{X, G(\eta, X)}\Vert p_{X, Y})
\end{equation}
if and only if $p_{X, G^*(\eta, X)}=p_{X,Y}$,  that is, $(X, G^*(\eta, X)) \sim (X,Y).$
\end{lemma}
This lemma is a direct consequence from the properties of the $f$-divergence.
Let
\begin{equation}
\label{dra}
r(z) = \frac{p_{X,G(\eta,X)}(z)}{p_{X,Y}(z)}, \ z \in \real^d \times \real^q.
\end{equation}
be the density ratio of $p_{X,G(\eta,X)}$ over  $p_{X,Y}.$
We only focus on the KL divergence below. By (\ref{KLc}), we have
\[
\Dbb_{\KL}(p_{X, G(\eta, X)}\Vert p_{X,Y}) =
\Ebb_{(X, \eta) \sim p_X p_{\eta}} [\log r(X, G(\eta, X))].
\]
Our goal is to minimize an empirical version of $\Dbb_{KL}(p_{X, G(\eta, X)}\Vert p_{X,Y})$ with respect to $G$. The minimizer will serve as an estimator of $G$.
The KL divergence depends on the unknown density functions
$p_{X, G(\eta, X)}$ and $p_{X, Y}$
only through the density ratio $r$ or the log-density ratio.
Denote the log-density ratio by  $D= \log  r.$
To estimate $G$, we will also need to estimate $D.$ We note that estimating density ratio is usually easier than estimating individual densities separately.
 The log-density ratio $D$ can be intuitively interpreted as a discriminator that quantifies the difference between the distributions of $(X, G(\eta, X))$ and $(X, Y)$.

Therefore, in our problem the loss function determined by the log-density ratio $D$  needs to be estimated along with the parameter of interest $G$. For this purpose, we consider the variational form of the KL divergence. The dual of $f(x)=x\log x$ is $f^*(t)= \exp(t-1)$ \citep{nguyen2010estimating}.
By Lemma \ref{lem2}, we can write the variational representation of the KL-divergence as
\begin{eqnarray}
\lefteqn{\Dbb_{\KL}(p_{X, G(\eta, X)}\Vert p_{X, Y})} \nonumber \\
&=& \sup_D \{\Ebb_{(X, \eta)\sim P_X  P_{\eta}} [D(X, G(\eta,X))] - \Ebb_{(X, Y) \sim P_{X,Y}} [\exp(D(X,Y)-1)]\} \nonumber \\
&=& \sup_D\{  \Ebb_{(X, \eta)\sim P_X P_{\eta}}[ D(X, G(\eta,X))]
- \Ebb_{(X, Y) \sim P_{X,Y}} [\exp(D(X,Y))]\} + 1, \label{KLdual}
\end{eqnarray}
where the second equality follows by change of variables from $D-1$ to  $D$ in the supremum operation. For the purpose of estimating $G$ by minimizing the KL divergence, we can ignore the constant $1$ in (\ref{KLdual}).
So we consider the criterion
\begin{eqnarray}\label{lossp}
\cL(G,D)=\Ebb_{(X, \eta)\sim P_X P_{\eta}}[ D(X, G(\eta,X))] - \Ebb_{(X, Y) \sim P_{X,Y}} [\exp(D(X,Y))].
\end{eqnarray}
The variational form is convenient since it is easy to obtain its empirical version when random samples are available. Then at the population level,  the target conditional generator $G^*$ and the target discriminator $D^*$ are characterized by the minimax problem
\begin{equation}
\label{objP}
(G^*, D^*) = \argmin_G \argmin_D \cL(G, D).
\end{equation}

Suppose that $\{(X_i, Y_i), i=1, \ldots, n\}$ are i.i.d. $P_{X, Y}$ and $\{\eta_i, i=1, \ldots, n\}$ are independently generated from  $P_{\eta}$.
We consider the following empirical version of $\cL(G,D)$:
\begin{equation}
\label{dualobj}
\widehat{\cL}(G, D) =
\frac{1}{n} \sum_{i=1}^n  D(X_i, G(\eta_i, X_i))
- \frac{1}{n}\sum_{i=1}^n \exp(D(X_i, Y_i)).
\end{equation}

We estimate $G$ nonparametrically using feedforward neural networks (FNN) \citep{Schmidhuber_2015} based on the objective function $\widehat{\cL}(G, D) $ in
(\ref{dualobj}).
We use two FNNs: the conditional  generator network $G_\vtheta$ with parameter $\vtheta$ for  estimating $G$ and a second network $D_\vphi$  with parameter $\vphi$ for estimating the discriminator $D$.
For any function $f(\vx): {\mathcal X} \to \real^d$,  denote $\|f\|_{\infty} = \sup_{\vx \in {\mathcal{X}}} \|f(\vx)\|$, where $\|\cdot\|$ is the Euclidean norm.

\begin{itemize}
\item The generator network $G_{\vtheta}$: let $\cG \equiv \cG_{\mathcal{H}, \mathcal{W}, \mathcal{S}, \mathcal{B}}$
be the set of ReLU  neural networks  $G_{\vtheta}: \real^m \times \real^d \rightarrow \real^{q}$ with parameter $\vtheta$,  depth $\mathcal{H}$, width $\mathcal{W}$, size $\mathcal{S}$, and $\|G_{\vtheta}\|_{\infty} \leq \mathcal{B}.$
Here the depth $\mathcal{H}$ refers to the number of hidden layers, so the network has $\cH+1$ layers in total. A $(\cH+1)$-vector $(w_0, w_1, \ldots, w_{\cH})$
specifies the width of each layer, where $w_0=d$ is the dimension of the input data and $w_{\cH}=q$ is the dimension of the output. The width $\cW=\max\{w_1, \ldots, w_{\cH}\}$ is the maximum width of the
hidden layers. The size $\cS=\sum_{i=0}^{\cH}[w_i\times (w_i+1)]$ is the total number of parameters in the network. For multilayer perceptrons with equal-width hidden layers except the output layer, we have
$\cS=\cW(m+1)+(\cW^2+\cW)(\cH-1)+\cW+q.$

\item  The discriminator network $D_{\vphi}$: Similarly, denote
	$\mathcal{D}\equiv \mathcal{D}_{\tilde{\mathcal{H}}, \tilde{\mathcal{W}}, \tilde{\mathcal{S}}, \tilde{\mathcal{B}}}$ as the set of ReLU  neural networks  $D_{\vphi}: \real^d\times \real^q\rightarrow \mathbb{R}$, with
	parameter $\vphi$, depth   $\tilde{\mathcal{H}}$,   width $\tilde{\mathcal{W}}$,  size $\tilde{\mathcal{S}}$, and $ \|D_{\vphi}\|_{\infty} \leq \tilde{\mathcal{B}}.$

\end{itemize}
Then $\vtheta$ and $\vphi$ are estimated by solving the empirical version of the
 minimax problem (\ref{objP}), that is,
\begin{equation}\label{losss}
(\hvtheta, \hvphi) = \argmin_{\vtheta}\argmax_{\vphi} \widehat{\cL}(G_{\vtheta}, D_{\vphi}).
\end{equation}
The estimated conditional generator is $\hat{G} = G_{\hvtheta}$ and the estimated discriminator is $\hat{D}=D_{\hvphi}.$  It is natural to compute $(\hvtheta, \hvphi)$ by alternately minimizing $L(\vtheta, \vphi)$ with respect to $\vtheta$ fixing $\vphi$ and maximizing $L(\vtheta, \vphi)$ with respect to $\vphi$ fixing $\vtheta$.
We provide the implementation details in Section \ref{implementation}.

\section{Weak convergence of conditional sampler}
In this section, we provide sufficient conditions under which $(X, \hG_n(\eta, X))$ converges in distribution to $(X, Y)$. This implies that for given $X=x$ with $p_X(x) > 0$, the conditional distribution of $\hG_n(\eta, x)$ given $X=x$ converges to the conditional distribution of $Y$ given $X=x$.
We focus on the case when $X$ and $Y$ are continuous-type random vectors.
We  establish a slightly stronger result by showing that the total variation norm
\[
\|p_{X, \hat{G}(\eta, X)}-p_{X, Y}\|_{L_1}
=\int_{\cX\times \cY}|p_{X, \hat{G}(\eta, X)}(x,y)-p_{X, Y}(x,y)| dxdy
\]
converges to zero.

Let ${\cL}(G,D)$ be defined in \eqref{lossp}.
For any measurable function $G: \real^m \times \real^d \mapsto \real^q$,  define
\begin{equation}\label{lossklp}
\mathbb{L}(G) 
=  \sup_{D} \cL(G,D).
\end{equation}
For a fixed $G$, let $p_{XG}$ be the joint density of $(X, G(\eta, X))$. Lemma \ref{lem2} implies that the optimal $D$ is
$
D^*(z) = \log ({p_{XG}(z)}/{p_{XY}(z)})=\log r(z).
$
Thus the optimal discriminator is the log-likelihood ratio serving as a critic of the resemblance between $p_{XY}$ and $p_{XG}$. Substituting this expression into (\ref{lossklp}) yields
$
\Lbb(G) = \Ebb_{(X, \eta) \sim P_XP_{\eta}} [\log r(X, G(\eta, X))].
$
Let
$
G^* \in  \argmin_{G} \Lbb(G).
$
We have $P_{(X, G^*(X, \eta))} = P_{X, Y}$ by Lemmas \ref{lem2} and \ref{KLM}.

We assume the following conditions.

\begin{itemize}
\item[(A1)]

The target conditional generator $G^*: \real^{m} \times \cX  \to \cY$ is continuous with $\|G^*\|_{\infty} \le C_0$ for some constant $0 < C_0 < \infty.$


\item[(A2)]

For any  $G \in \cG\equiv \cG_{\mathcal{D}, \mathcal{W}, \mathcal{S}, \mathcal{B}}$,
$r_G(z) = {p_{X,G(\eta,X)}(z)}/{p_{X,Y}(z)}: \cX\times \cY \to \real$ is continuous and
$0<C_1\leq r_G(z)\leq C_2$ for some constants $0 < C_1 \le C_2 < \infty.$

\end{itemize}

We also make the following assumptions on the network parameters
of the conditional generator $G_{\vtheta}$  and the discriminator  $D_{\vphi}$.

\begin{itemize}
\item[(N1)] The network parameters of $\cG$ satisfies
	\begin{equation*}
		\mathcal{H}\mathcal{W}\to\infty \quad {\rm and} \quad \frac{\mathcal{B}\mathcal{S}\cH\log(\mathcal{S})\log n}{n}\to 0,
\text{ as } n \to \infty.
	\end{equation*}

\item[(N2)]
The network parameters of $\cD$ satisfies
	\begin{equation*}
		\tilde{\mathcal{H}}\tilde{\mathcal{W}}\to\infty \quad {\rm and} \quad \frac{\tilde{\mathcal{B}}\tilde{\cS}\tilde{\cH}
\log(\tilde{\cS})\log n}{n}\to 0, \text{ as } n \to \infty.
	\end{equation*}
\end{itemize}

\begin{theorem}
\label{consistency}
Suppose that the assumptions (A1) and (A2) hold. If the network parameters of $\cG$ and $\cD$ satisfies the specifications (N1) and (N2), then
\begin{equation}
\label{converg}
\Ebb_{(X_i,Y_i,\eta_i,)_{i=1}^n}  \|p_{X, \hat{G}_{\vtheta}(\eta, X)}-p_{X, Y}\|_{L_1}^2 \rightarrow 0,
\ \text{ as } n \to \infty.
\end{equation}

 \end{theorem}

A direct corollary of Theorem \ref{consistency} is the following convergence result in terms of the conditional density functions.
\begin{corollary}
\label{consistencyb}
Suppose that the assumptions (A1) and (A2) hold and the network parameters of $\cG$ and $\cD$ satisfies the specifications (N1) and (N2). Then
\begin{equation*}
\label{convergb}
 \Ebb_{X\sim P_X} \left[\int_{\cY} \left|p_{\hat{G}_{\vtheta}(\eta, X)}(y| X)-p_{Y|X}(y| X)\right| dy\right] \rightarrow_P 0,  \text{ as } n \to \infty.
\end{equation*}
\end{corollary}

Theorem \ref{consistency} and Corollary \ref{consistencyb} provide strong theoretical support for the proposed method under mild conditions. They are proved using the empirical process method \citep{bartlett2002rademacher, bartlett2019} and the recent results on approximating smooth functions by deep neural networks \citep{shen2019deep}.
The main challenge of the proof is that the objective function is a minimax process indexed by two classes of neural networks.
Details are given in the supplementary material.

Conditions (A1) and (A2) are mild regularity conditions that are often assumed in nonparametric estimation problems.
Conditions (N1) and (N2)
concern the depths, widths and sizes of the generator and the discriminator networks. For
the generator
These conditions require that the size of the network increases with the sample size,  the product of the depth and the width increases with the sample size. We note that the conditions are flexible with respect to the network architecture. In particular, they allow either the depth or the width remain fixed. For example, we can have a deep network with fixed width or a wide network with fixed depth. A restriction of the conditions is that they require the network size to be smaller than the sample size. This restriction stems from the use of empirical process theory \citep{vw1996, bartlett2019, bartlett2002rademacher} to control the stochastic error of the estimated generator and discriminator.

In nonparametric regression, there has been much recent work on convergence analysis of nonparametric estimators using deep neural networks.
Two types of assumptions on the underlying model have been used in the analysis.
The first type of assumptions postulates that
the regression function has a compositional structure so that the intrinsic dimension of the function is lower than the ambient dimension  \citep{bauer2019deep, kohler2020rate, schmidt2020nonparametric, shen2021quantile}.
The second type assumes that
the distribution of $X$ is supported on a lower-dimensional manifold
\citep{chen2019nonparametric, nakada2019adaptive, jiao2021deep, shen2021robust}.
These works also require the size of the neural network used in the nonparametric regression to be smaller than the sample size to ensure the consistency of the estimators.

In the current problem of sampling from conditional distributions,
convergence analysis is substantially more difficult than that in nonparametric regression. The main reason is that in the current setting, optimization is a minimax problem that leads to an \textit{estimated nonparametric loss function}, i.e.,  there is a second neural network involved for estimating the discriminator in the dual form of KL divergence, in addition to the neural network for estimating the conditional generator. In comparison,
nonparametric regression with a given loss function such as least squares loss
only has a single neural network for estimating the regression function.

We now give a high-level description of the proof for Theorem \ref{consistency}, the details are provided in the supplementary material.
Lemma \ref{KLM} implies that $\mathbb{L}(G^*) = 0.$
For notational simplicity, write $\hat{G}=\hat{G}_{\vtheta}.$
By Pinsker's inequality \citep{npe2008}, we have
\begin{equation}\label{errl1}
 \|p_{X, \hat{G}(\eta, X)}-p_{X, Y})\|_{L^1}^2 \leq 2 (\mathbb{L}(\hat{G}) - \mathbb{L}(G^*)).
 \end{equation}
So it suffices to show that the right side in \eqref{errl1}
converges to zero in expectation.
By the definition of $\Lbb(G)$, we can write the excess risk as
\[
\mathbb{L}(\hat{G})-\mathbb{L}(G^{*})
=\sup_{D} \cL(\hat{G},D)-\sup_{D} \cL(G^*,D).
\]
A key step in the proof is the following decomposition of the excess risk.
For any	$\bar{G}\in \cG$, we decompose the right side in \eqref{errl1} as:
\begin{eqnarray}
\mathbb{L}(\hat{G})-\mathbb{L}(G^{*}) 
& = &\sup_{D} \cL(\hat{G},D)  -  \sup_{D \in \cD} \cL(\hat{G},D) \nonumber\\
& &+   \sup_{D \in \cD} \cL(\hat{G},D) - \sup_{D \in \cD} \widehat{\cL}(\hat{G},D) \label{ep1}\\
& &+ \sup_{D \in \cD} \widehat{\cL}(\hat{G},D) - \sup_{D \in \cD} \widehat{\cL}(\bar{G},D) \label{neg1}\\
& &+ \sup_{D \in \cD} \widehat{\cL}(\bar{G},D) - \sup_{D \in \cD} \cL(\bar{G},D) \label{ep2}\\
& &+\sup_{D \in \cD} \cL(\bar{G},D) - \sup_{D} \cL(\bar{G},D)\label{neg2}\\
& &+ \sup_{D} \cL(\bar{G},D)-\sup_{D} \cL(G^*,D).\nonumber
\end{eqnarray}
Since the terms in \eqref{neg1} and \eqref{neg2} are nonpositive,
and the terms in \eqref{ep1} and  \eqref{ep2}  are smaller than
$\sup_{D\in \cD, G\in \cG}|\cL(G,D)-\widehat{\cL}(G,D)|$,
we have
\begin{eqnarray*}
\mathbb{L}(\hat{G})-\mathbb{L}(G^{*})
&\leq & \sup_{D} \cL(\hat{G},D)  -  \sup_{D \in \cD} \cL(\hat{G},D) + 2 \sup_{D\in \cD, G\in \cG}|\cL(G,D)-\widehat{\cL}(G,D)|\nonumber\\
& &+ \sup_{D} \cL(\bar{G},D)-\sup_{D} \cL(G^*,D).\nonumber
\end{eqnarray*}
Note that $\bar{G}$ is arbitrary. By taking infimum with respect to  $\bar{G}$ over $\cG$ on both sides of the above display,  we obtain
\begin{align}
\label{errdec}
\mathbb{L}(\hat{G})-\mathbb{L}(G^{*}) \le
\Delta_1 + \Delta_2 + \Delta_3,
\end{align}
where
\begin{eqnarray*}
\Delta_1&=&\sup_{D} \cL(\hat{G},D)  -  \sup_{D \in \cD} \cL(\hat{G},D),\\
\Delta_2&= &2 \sup_{G \in \cG, D\in \cD}|\cL(G,D)-\hat{\cL}(G,D)|,\\
\Delta_3&=&\inf_{\bar{G}\in \cG} [\mathbb{L}(\bar{G})-\mathbb{L}(G^*)].
\end{eqnarray*}
The first and the third terms $\Delta_1$ and $\Delta_3$
are the approximation errors and
the second term $\Delta_2$ is the statistical error.
In the supplementary material, we derive (\ref{errdec}) and show that these error terms converge to zero.

\section{Implementation}
\label{implementation}

We describe the implementation of GCDS. For training the generator $G_{\vtheta}$ and the discriminator $D_{\vphi}$, we use the rectified linear unit (ReLU) as the activation function in $G_{\vtheta}$ and $D_{\vphi}.$
%
%
We train the discriminator and the generator iteratively by updating $\vtheta$ and $\vphi$ alternately as follows:
\begin{itemize}
 \setlength\itemsep{-0.2 cm}
	\item[(a)] Fix  $\vtheta$,  update the discriminator by ascending the stochastic gradient of the loss $\eqref{dualobj}$ with respect to $\vphi$.

	\item[(b)] Fix  $\vphi$,  update the generator by descending the stochastic gradient of the loss $\eqref{dualobj}$ with respect to $\vtheta$.
\end{itemize}

The  training process is described below.

\bigskip\noindent
\textbf{Algorithm}: Training GCDS

\noindent
\textbf{Input:}
(a) Pairs $\{(X_i,Y_i), \ i = 1,\dots,n\};$
(b)  Samples $\{\eta\}_{i=1}^n$ from $P_{\eta}$

\noindent
\textbf{Output}: Conditional generator
$G_{\hvtheta}$ and discriminator $D_{\hvtheta}$

\noindent
\textbf{While not converged do}
\begin{itemize}
 \setlength\itemsep{-0.1 cm}
\item  Compute
$\tilde{Y}_i = G_\vtheta(\eta_i, X_i),\ i = 1,2,\dots,n.$
Let $S_1= \{(X_i, Z_i, V_i) = (X_i, Y_i,1), \ i=1,\dots,n\}$ and $S_2 = \{(X_i, Z_i,V_i) = (X_{i-n}, \tilde{Y}_i,-1), \ i = n+1,\dots,2n\}.$
\item Randomly select ${B}/{2}$ samples from $S_1$ and another ${B}/{2}$ samples from $S_2$.
Denote the subscripts of the selected samples by $\{b_i:i=1,\dots,B\}.$
\item
Update $D_\vphi$ by ascending its stochastic gradient:
\[
\nabla_\vphi \Big\{ \frac{1}{B}\sum_{i = 1}^B [D_\vphi(X_{b_i}, Z_{b_i}) \mathbbm{1}_{\{V_{b_i} = -1\}}
- \exp(D_\vphi(X_{b_i}, Z_{b_i}))\mathbbm{1}_{\{V_{b_i} = 1\}}]\Big\}.
\]

\item Randomly select B samples from $\{(X_i,Y_i), \ i = 1,\dots,n\}.$  Denote the subscripts of the selected samples by $\{b_i:i=1,\dots,B\}$\;

\item Update $G_\vtheta$ by descending its stochastic gradient:
		$$\nabla_\vtheta \left\{ \frac{1}{B}\sum_{i=1}^B D_\vphi(X_{b_i}, G_\vtheta(\eta_{b_i}, X_{b_i})) \right\}.
		$$
\end{itemize}
\textbf{End while}

We implement the GCDS algorithm in
TensorFlow \citep{abadi2016tensorflow}.

\section{Numerical experiments}
\label{experiments}
In this section, we carry out numerical experiments to assess the performance
of GCDS. We use both simulated and real datasets in the experiments.
In addition to the results reported in this section, additional numerical results are given in the online supplementary material, including results from experiments evaluating  how the performance of GCDS depends on the network architecture,
the dimension $m$ of the noise vector $\eta$ and the sample size $n.$

\subsection{Simulation studies}
We conduct simulation studies to evaluate the finite sample performance of GCDS. We also compare it with several existing conditional density estimation methods, including the
nearest neighbor kernel conditional density estimation (NNKCDE,
\cite{dalmasso2020conditional}), the conditional kernel density estimation (CKDE, implemented in the \texttt{R} package \texttt{np}, \cite{hall2004cross}), and the basis expansion method FlexCode \citep{izbicki2017converting}). We implement GCDS in
TensorFlow \citep{abadi2016tensorflow} and use the stochastic gradient descent algorithm  Adam \citep{kingma2015adam} in training the neural networks.
%
%
We use the conditional distributions based on the following models in the simulation studies.

\begin{enumerate}
 \setlength\itemsep{0.0 cm}
\item  (M1). A nonlinear model with an additive error term:\\
$Y =  X_1^2 + \exp(X_2 + X_3/3) + \sin(X_4 + X_5) + \varepsilon,\  \varepsilon\sim N(0,1).$

\item  (M2). A model with  an additive error term whose variance depends on the predictors:
$Y =
 X_1^2 + \exp((X_2 +X_3/3)) + X_4 - X_5 + (0.5+ X_2^2/2 +X_5^2/2)\times \varepsilon,\   \varepsilon\sim N(0,1).$

\item  (M3). A model with a multiplicative non-Gassisan error term:\\
$Y =
(5 +X_1^2/3 + X_2^2 + X_3^2+X_4+X_5) *\exp(0.5\times\varepsilon),$  where\\
$\varepsilon\sim \mathbb{I}_{\{U<0.5\}}\times N(-2,1)+\mathbb{I}_{\{U>0.5\}}\times N(2,1)\  \text{with}\ U\sim \text{Uniform}(0,1), X \in \real^{30}.$


\item  (M4).  A mixture of two normal distributions:\\
$ Y =
\mathbb{I}_{\{U<0.5\}}N(-X_1, 0.25^2) + \mathbb{I}_{\{U>0.5\}}N(X_1, 0.25^2),$ where
$ U \sim \text{Uniform}(0,1).$

\end{enumerate}
In each of  the models above, the covariate vector  $X$ is generated from standard multivariate normal distribution.

The neural networks used in the simulations are specified as follows. For models (M1)-(M3),
the generator network has 1 hidden layer with width 50, and the discriminator has 2 hidden layers with widths $(50, 25);$ for models (M4), the generator network has 2 hidden layers with widths (40, 15), and the discriminator has 2 hidden layers with widths $(50, 25).$
For the values of $m$, the dimension of the noise random vector $\eta$,  we set $m=3$ for models (M1)-(M3), and $m=4$ for model (M4).

For the conditional density estimation method NNKCDE, the tuning parameters
are chosen using cross-validation. The bandwidth of the conditional kernel density estimator CKDE is chosen by the rule-of-thumb using the standard formula $h_j = 1.06\sigma_j n^{-1/(2*K+J)}$ where $\sigma_j$ is a
measure of spread of the $j$th continuous variable defined as $\min(SD, IQR/1.349)$, n the number of observations, $K$ the order of the kernel, and $J$ the number of continuous variables.
The basis expansion based method FlexCode uses 
Fourier basis. The maximum number of bases is 40 and the actual number of bases is selected using cross-validation.

We calculate the mean squared error (MSE) of the estimated conditional mean $E(Y|X)$ and the estimated conditional standard deviation $\text{SD}(Y|X)$.
We use a test data set $\{x_1,\dots, x_{k}\}$ of size $k=2000$ . The MSE of the estimated conditional mean is
$
\text{MSE(mean)}= (1/k)\sum_{i=1}^{k}[\hat E(Y|X=x_i) - E(Y|X=x_i)]^2.
$
For GCDS, the estimate of $E(Y|X=x)$ is based on (\ref{mont}) using Monte Carlo. For other methods, the estimate is calculated by numerical integration $\hat E(Y|x) = \int yf(y|x)dy$ using 1000 subdivisions.
Similarly, the MSE of the estimated conditional standard deviation is
$ \text{MSE(sd)} = (1/k)\sum_{i=1}^{k}[\hat{\text{SD}}(Y|X=x_i) -\text{SD}(Y|X=x_i)]^2.
$

\begin{table}[H]
	\centering
	\begin{tabular}{cC{1.5cm}C{2.7cm}C{2.7cm}C{2.7cm}C{2.7cm}} \hline
& &  GCDS & NNKCDE & CKDE & FlexCode \\ \hline
		\multirow{2}{*}{M1}&Mean &  \textbf{0.259}(0.015)  & 1.367(0.010) & 0.491(0.024) & 0.610(0.008) \\
		&SD &  \textbf{0.022}(0.004) & 0.258(0.004) & 0.233(0.005) & 0.170(0.007) \\ \hline
&Mean &  \textbf{0.312}(0.017)  & 4.668(0.046) & 1.707(0.060) & 2.408(0.063) \\
\multirow{-2}{*}{M2}&SD &  \textbf{0.247}(0.012) & 0.793(0.008) & 0.857(0.017) & 2.384(0.602) \\ \hline
		\multirow{2}{*}{M3}&Mean & \textbf{3.377}(0.196) & 4.926(0.080) & 39.084(0.929) & 9.015(0.341) \\
		&SD &  \textbf{2.082}(0.126) & 8.131(0.235) & 15.70(0.488) & 11.53(1.140) \\ \hline
&Mean &  0.016(0.003) & \textbf{0.004}(0.001) & 0.063(0.002) & 0.006(0.002) \\
\multirow{-2}{*}{M4}&SD & \textbf{0.027}(0.005) & 0.131(0.001) & 0.076(0.001) & 0.046(0.001) \\ \hline
	\end{tabular}
	\caption{Mean squared error(MSE) of the estimated conditional mean, the estimated standard deviation and the corresponding simulation standard errors (in parentheses).
The smallest MSEs are in bold font.}
	\label{tab:point}
\end{table}

For GCDS, we first generate $J$ samples $\{\eta_j:j=1, \ldots, J\}$ from the reference distribution $P_{\eta}$ and calculate
conditional samples  $\{\hat G(\eta_j, x_i), j = 1,\dots, J\}.$
We take $J=10,000$.
The estimated conditional standard deviation is calculated as the sample standard deviation of the conditional samples.
The estimated conditional standard deviation of other methods are computed by numerical integration
$\hat{\text{SD}}(Y|x_i) = \sqrt{\int [y-\hat E(Y|x_i)]^2f(y|x_i)dy}$ using 1000 subdivisions.

We repeat the simulations 10 times.  The average MSEs and simulation standard errors are summarized in Table \ref{tab:point}. We see that, comparing with CKDE and FlexCode and NNKCED,  GCDS has the smallest MSEs for estimating conditional mean and conditional SD in most cases.

\begin{figure}[H]
	\centering
\includegraphics[width=4.6 in, height=2.4 in]
{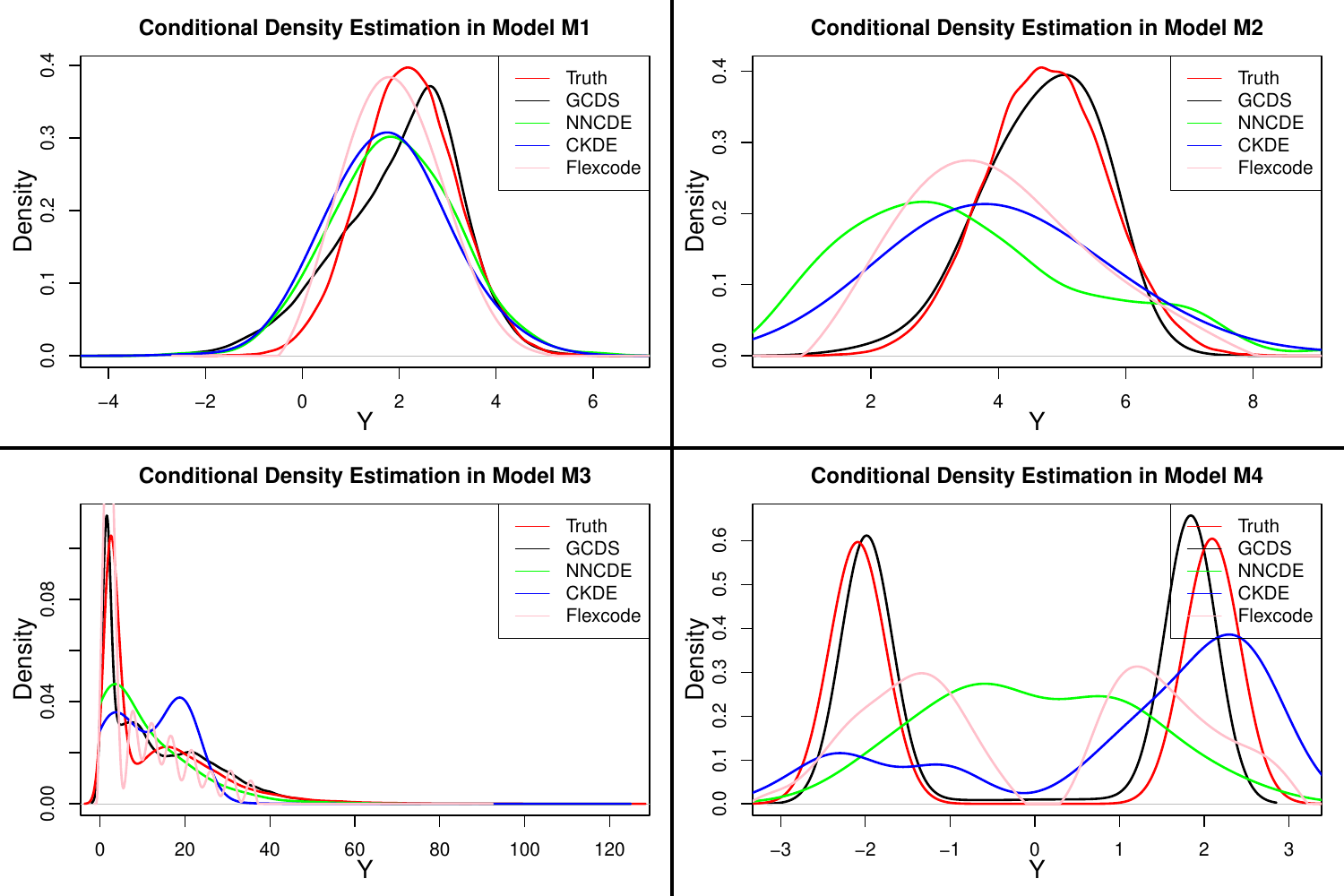}
\caption{Comparison of density estimation in models (M1) to (M4).
The conditional density function corresponding to GCDS is estimated using the samples generated from GCDS with kernel smoothing.}
	\label{fig:pdf1234}
\end{figure}

In Figure \ref{fig:pdf1234}, we display the estimated conditional density functions for a randomly generated value of $X$.
The true conditional distributions of $Y|X$  for models (M1) to (M4) are:  (M1), $N(2.19,1)$; (M2), $N(4.75, 0.96^2)$; (M3): the mixture of half $7.42\times \text{log-normal}(-1, 0.5^2)$ and half $7.42\times \text{log-normal}(1, 0.5^2)$; (M4): the mixture of half $N(-2.09,0.25^2)$ and half $N(2.09, 0.25^2)$.
The conditional density function corresponding to GCDS is estimated based the samples generated based on
GCDS using kernel smoothing. This plot shows that GCDS yields better conditional density estimates than CKDE, NNCDE and FlexCode.


\subsection{The abalone dataset}
The abalone dataset is available at UCI machine learning repository \citep{Dua:2019}.
It contains the number of rings of abalone and other physical measurements. The age of abalone is determined by cutting the shell through the cone, staining it, and counting the number of rings through a microscope, a time-consuming process. Other measurements, which are easier to obtain, are used to predict the number of rings that determines the age.
This dataset contains 9 variables. They are \textit{sex, length, diameter, height, whole weight, shucked weight, viscera weight, shell weight} and \textit{rings}.
Except for the categorical variable \textit{sex}, all the other variables are continuous. The variable \textit{sex} codes three groups: female, male and infant, since the gender of an infant abalone is not known. The sample size is 4177. We use 90\% of the data for training and 10\% of the data  as the testing set.
The neural networks used in the analysis are specified as follows: the generator network is a fully connected network with 2 hidden layers with widths 50 and 20; the discriminator networks
is a fully connected network with 2 hidden layers with widths 50 and 25. The dimension of the noise vector of the noise vector is set to be $m=5.$

\bigskip
\begin{figure}[H]
	\centering
	\includegraphics[width=4.0 in, height=2.0 in]{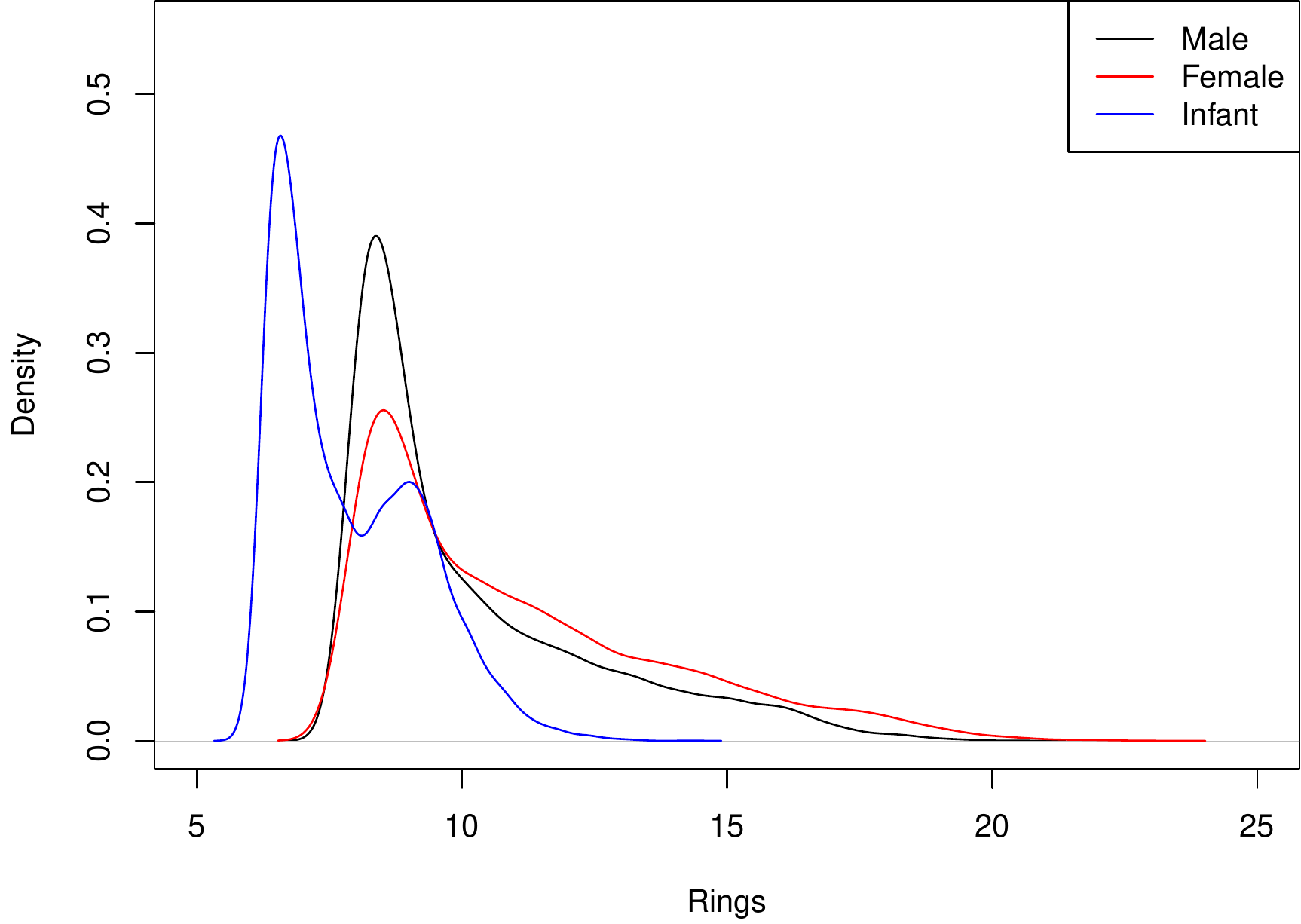}
	\caption{Estimated conditional densities for the female, male and infant groups in the abalone dataset. Each line represents the kernel conditional density estimation based on the samples generated using GCDS given the group average values of the covariates.}
	\label{fig:abaringgroup}
\end{figure}

We take \textit{rings} as the response $Y \in \real$ and the other measurements as the covariate vector $X \in \real^9$.
Figure \ref{fig:abaringgroup} shows the estimated conditional density based on the training dataset for 3 groups: female, male and infant,
at the value of the group means of the remaining covariates. We see that the values of \textit{rings} of the infant group are smaller than those of the female and male groups. The female abalones tend to have slightly higher numbers of rings than  male abalones. In addition,  the conditional distributions are skewed to the right for all the three groups.

\begin{figure}[H] 
	\centering
	\includegraphics[width=5.6 in, height=2.0 in]{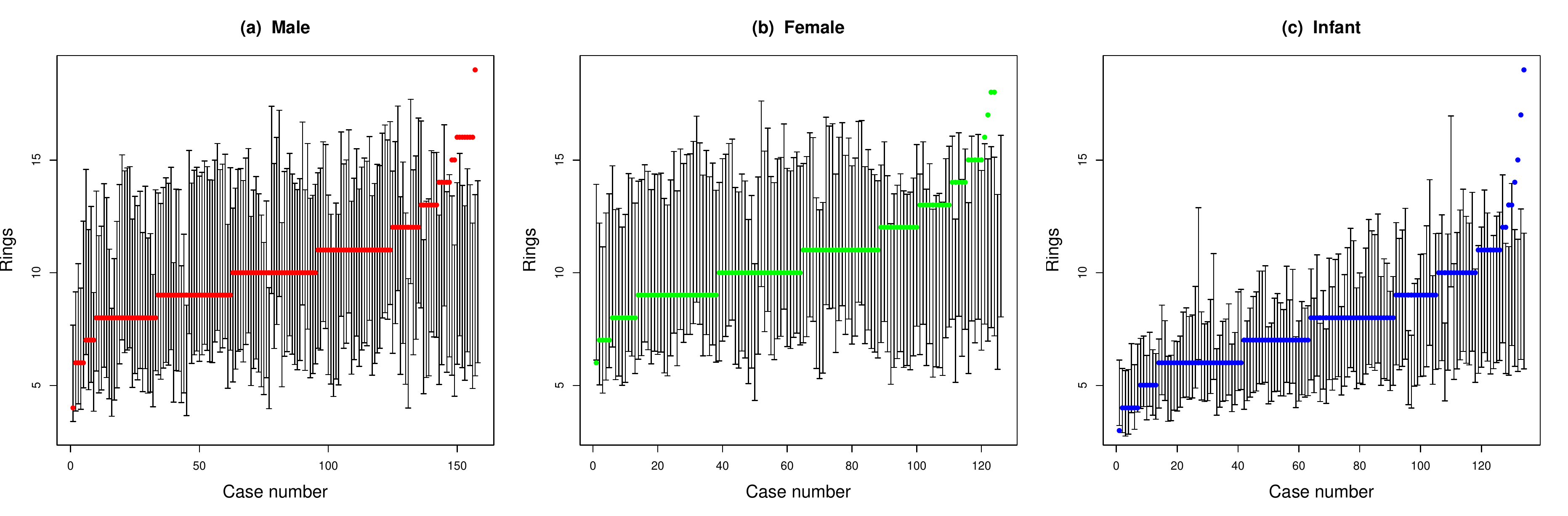}
	\caption{The prediction intervals for the testing set. All 418 abalones in the testing set are divided into three groups, (a) male, (b) female,  and (c) infant.}
	\label{fig:abaringsingle}
\end{figure}

To examine the prediction performance of the estimated conditional density, we construct the 90\% prediction interval for the number of rings of each abalone in the testing set. The prediction intervals are shown in Figure \ref{fig:abaringsingle}. The actual number of rings are plotted as a solid dot. The actual coverage for all 418 cases in the testing set is 89.71\%, close
to the nominal level of 90\%. The numbers of rings that are not covered by the prediction intervals are the largest ones in each group.

\subsection{MNIST handwritten digits}
We now illustrate the application of GCDS to high-dimensional data problems and demonstrate that it can easily handle the models when either of both of $X$ and $Y$ are high-dimensional.
The data example we use is
the MNIST handwritten digits dataset \citep{mnist},  which contains 60,000 images for training and 10,000 images for testing. The images are stored in $28\times28$ matrices with gray color intensity from 0 to 1. Each image is paired with a label in $\{0,1\dots,9\}$.
We use GCDS to perform two tasks: generating images from labels and reconstructing the missing part of an image.

\noindent
\textbf{Generating images from labels }
We generate images of handwritten digits given the label. In this problem,  the predictor $X$ is a categorical variable representing the ten digits: $\{0,1,\dots,9\}$ and the response $Y$ represents $28 \times 28$ images.
We use one-hot vectors in $\real^{10}$ to represent these ten categories. So the dimension of $X$ is 10
and the dimension of  $Y$ is $28\times 28=784.$
The response $Y\in [0,1]^{28\times28}$ is a matrix representing the intensity values. For the discriminator $D$, we use a convolutional neural network (CNN) with 3 convolution layers with 128, 256, and 256 filters to extract the features of the image and then concatenate with the label information (repeated 10 times to match the dimension of the features). The concatenated information is sent to a fully connected layer and then to the output layer. For the generator $G$, we concatenate the label information with random noise of dimension 100. Then it is fed to a CNN with 3 deconvolution layers with 256, 128, and 1 filters. 

\begin{figure}[H]
	\centering
\includegraphics[width=4.6in, height=2.2in]{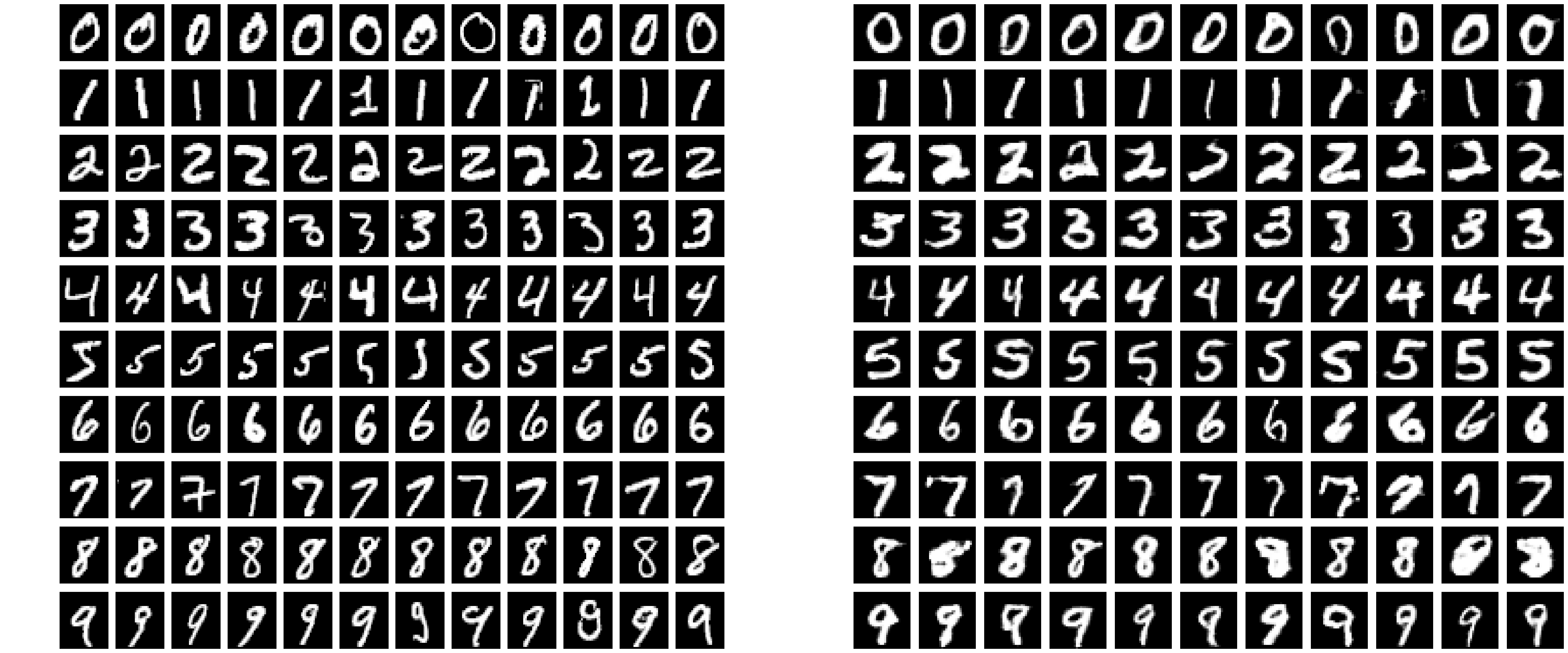}
	\caption{MNIST dataset: real images (left panel) and generated images given the labels (right panel).}
	\label{fig:mnist_xlabel}
\end{figure}

Figure \ref{fig:mnist_xlabel} shows the real images (left panel) and generated images (right panel).
We see that the generated images are similar to the real images and it is hard to distinguish the generated ones from the real images.
Also, there are some differences in the generated images, reflecting the random variations in the generating process.

\noindent
\textbf{Reconstructing missing part of an image }
We now illustrate using GCDS to reconstruct an image when part of the image is missing with the MNIST dataset.
Suppose we only observe ${1}/{4}, {1}/{2}$ or ${3}/{4}$ of an image and would like to reconstruct the missing part of the image.
For this problem, let $X$ be the observed part of the image
and let $Y$ be the missing part of the image. Our goal is to reconstruct $Y$ based on $X$.
For the discriminator, we use two convolutional networks to process  $X$ and $Y$ separately. The filters are then concatenated together and fed into another convolution layer and fully-connected layer before output. For the generator, $X$ is processed by a fully-connected layer followed by 3 deconvolution layers.

\begin{figure}[H]
	\centering
	\includegraphics[width=4.6 in, height=2.2in]{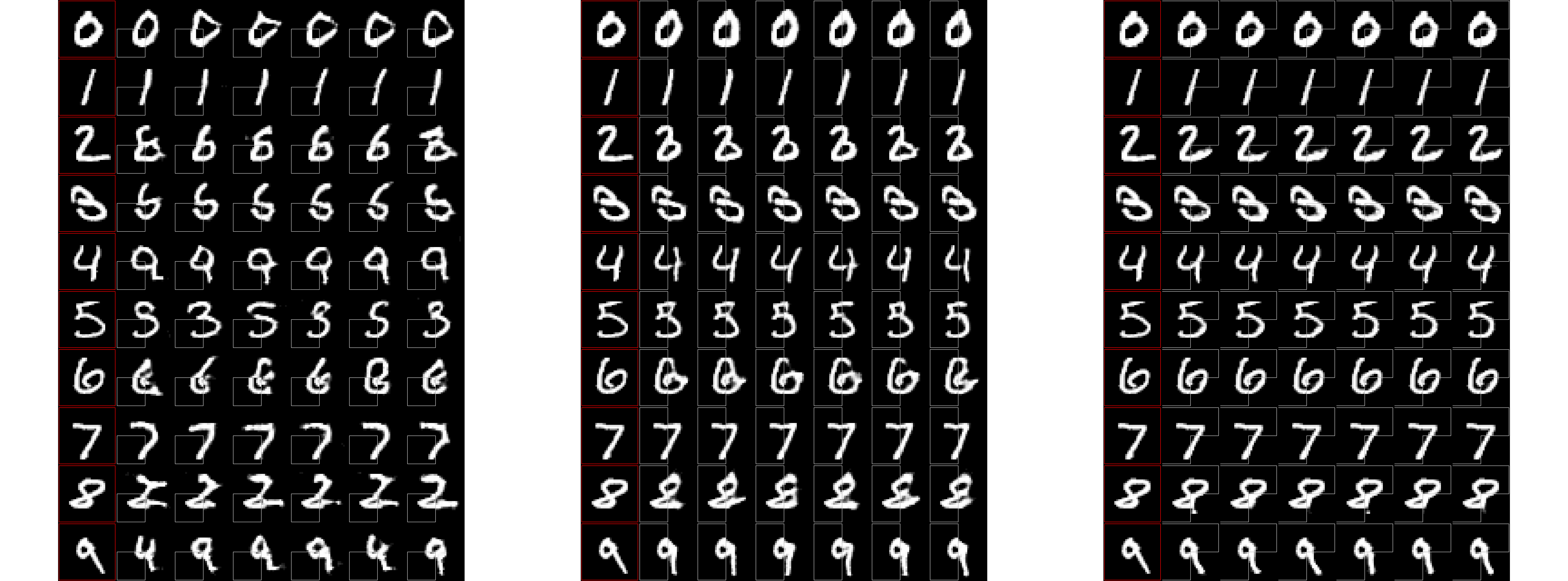}
	\caption{Reconstructed images given partial image in MNIST dataset. The first column in each panel consists of the true images, the other columns give the constructed images.
In the left panel, the left lower 1/4 of the image is given;  in the middle panel, the left 1/2 of the image is given; in the right panel, 3/4 of the image is given.}
	\label{fig:mnist_ximage}
\end{figure}

In Figure \ref{fig:mnist_ximage}, three plots from left to right corresponds to the situations when ${1}/{4}, {1}/{2}$ and ${3}/{4}$ of an image are given. In each subplot, the first column
contains the true images in the testing set. The gray boxes show the given areas. Each row contains six reconstructions of the image. The digits ``0'', ``1'' and ``7'' are easy to reconstruct. Even when only 1/4 of their images are given, GCDS can correctly reconstruct them. The other digits are more difficult. If only 1/4 of their images are given, it is impossible to reconstruct them.  However, as the given area increases from 1/4 to 1/2 and then 3/4 of the images, GCDS is able to reconstruct the images correctly, and the reconstructed images become less variable and more similar to the true image. For example, for the digit ``2'', if only the left lower 1/4 of the image is given, the reconstructed images tend to be incorrect; the reconstruction is only successful when 3/4 of the image is given.

\section{Conclusion}
In this paper we propose GCDS, a generative approach to sampling from a conditional distribution. We provide theoretical support for GCDS by showing that the conditional generator converge in distribution  to the underlying target conditional distribution under mild conditions. Our numerical experiments demonstrate that it works well in a variety of situations from the standard nonparametric conditional density estimation problems to more complex image data problems.

Several questions deserve further investigation.
First, it would be interesting to derive the convergence rate of the sampling distribution that strengthens the consistency result in Theorem \ref{consistency}
and provide conditions under which
the number of coefficients in the deep neural network is allowed to be greater than the sample size.
Second, while the conditional generator provides all the information of the conditional distribution, it is still useful to obtain an estimate of the functional form of the conditional density. How to obtain a good estimator of the conditional density function in the present framework is an open question, especially when the dimension of $(X, Y)$ is high.
Finally, as a proof of concept we demonstrated that GCDS yields reasonable results
for some simple image analysis tasks with the MNIST dataset.
It would be interesting to apply GCDS to more complex image analysis problems.
We intend to study these challenging problems in the future.

\section*{Supplementary material} Additional numerical experiment results and technical details are provided in the supplementary material.

\section*{Acknowledgements}
The work of X. Zhou and J. Huang is supported in part by the U.S. National Science Foundation grant DMS-1916199.
The work of Y. Jiao is supported by the National Science Foundation of China grants No.11871474 and No.61701547. The work of J. Liu is supported by the Duke-NUS Graduate Medical School WBS: R-913-200-098-263 and MOE2016-M2-2-029 from the Ministry of Education, Singapore.


\spacingset{1.5} 
\appendix
\setcounter{equation}{0}  
\renewcommand{\theequation}{C.\arabic{equation}}
\setcounter{table}{0}
\renewcommand{\thetable}{A.\arabic{table}}
\setcounter{figure}{0}
\renewcommand{\thefigure}{A.\arabic{figure}}
\setcounter{equation}{0}  
\setcounter{lemma}{0}
     \renewcommand{\thelemma}{\Alph{section}.\arabic{lemma}}
\setcounter{theorem}{0}
     \renewcommand{\thetheorem}{\Alph{section}.\arabic{theorem}}

\bigskip\bigskip
\noindent
\textbf{\Large Supplementary material}

In this Supplementary Material, we provide additional numerical experiment results, present the variational forms of three commonly used $f$-divergences, explain the difference between
the $f$-divergence definition used in this paper and that in \cite{nguyen2010estimating},  and prove Lemma \ref{NOLemma}, Lemma \ref{lem2}, and Theorem \ref{consistency}.


\section{Additional numerical experiments}

\subsection{Additional simulation results}
\subsubsection{Conditional quantile estimation}
We consider the conditional distributions based on the following models as given in (M1)-(M4) in the main text:
\begin{enumerate}
 \setlength\itemsep{-0.2 cm}
\item Model (M1):
$$Y = 
X_1^2 + \exp(X_2 + X_3/3) + \sin(X_4 + X_5) + \varepsilon,\ \  \varepsilon\sim N(0,1).$$
This is a nonlinear model with an additive error term.

\item Model (M2):
$$Y = 
 X_1^2 + \exp((X_2 +X_3/3)) + X_4 - X_5 + (0.5+ X_2^2/2 +X_5^2/2)\times \varepsilon,\ \  \varepsilon\sim N(0,1).$$
This model has an additive error term whose variance depends on the predictors.

\item Model (M3):
$$Y = 
(5 +X_1^2/3 + X_2^2 + X_3^2+X_4+X_5) *\exp(0.5\times\varepsilon)$$
$$\varepsilon\sim \mathbb{I}_{\{U<0.5\}}\times N(-2,1)+\mathbb{I}_{\{U>0.5\}}\times N(2,1)\  \text{where}\ U\sim \text{Uniform}(0,1),$$ with $p = 30$.
This model has a multiplicative non-Gassisan error term.

\item Model (M4):
$$ Y = 
\mathbb{I}_{\{U<0.5\}}N(-X_1, 0.25^2) + \mathbb{I}_{\{U>0.5\}}N(X_1, 0.25^2),$$
$$U \sim \text{Uniform}(0,1).$$
In this example, the conditional distribution is a mixture of two normal distributions.
\end{enumerate}
In each of  the models above, the covariate vector  $X$ is generated from standard multivariate normal distribution.

In models (M1) to (M3), the conditional generator $G$ is parameterized by a
{one-layer} neural network; in model (M4),  it is parameterized by a two-layer fully connected neural network. The log-density ratio function $D$ is parameterized by a
{two-layer} fully connected neural network. We use the stochastic gradient descent algorithm as implemented in Adam in training the neural networks.

We now examine the estimation of the conditional quantiles for a given $X=x$ defined as
$q_\tau = F_{Y|x}^{-1}(\tau|x)$ for $0 < \tau < 1$.

The conditional quantiles can be estimated easily using Monte Carlo.  For a given $X=x$, we generate $J$ a random sample of i.i.d.  $\eta_1, \ldots, \eta_{J}$ from the reference distribution $P_{\eta}$.
Then the estimated quantiles of the conditional distribution at $X=x$  is calculated based on
$\hat G(x,\eta_j), j=1, \ldots, J.$ We take $J=10,000.$
For FlexCode, KCDE and NNCDE, we solve $\hat F_{Y|x}(q) = \tau$ for $q$ to obtain the estimated $\tau$th conditional quantile.
We use a test data set $\{x_1,\dots, x_{k}\}$ of size $k=2000$ .
We consider the MSE  of the estimated conditional quantile of level $\tau$ defined as
$$ \text{MSE}(\tau) = \frac{1}{k}\sum_{i=1}^{k}[\hat{F}_{Y|X}^{-1}(\tau|X=x_i) - F_{Y|X}^{-1}(\tau|X=x_i)]^2.$$
We consider five levels of $\tau = 0.05, 0.25, 0.5, 0.75, 0.95$.

The simulation is repeated 10 times. The average MSEs and the corresponding simulation standard errors are included in Table \ref{tab:conditionalquantile1}.
We see all methods have larger MSEs in the tail area than in the center. This problem is more severe for KCDE and NNCDE. In model (M4), FlexCode fits the data near the first mode by using more basis functions, which leads to over-fitting in right tail area. Thus the MSE of FlexCode is small when $\tau = 0.05, 0.25$ but large when $\tau =0.75, 0.95$. In model (M4), GCDS has a larger MSE at $\tau = 0.5$. The reason is that the estimate of GCDS is a bimodal conditional distribution with few observations around
zero, which can be seen from Figure \ref{fig:pdf1234} that the true conditional distribution has low density around zero.
\begin{table}[htbp]
	\centering
	\rowcolors{2}{gray!10}{white}
	\begin{tabular}{cC{2.7cm}C{2.7cm}C{2.7cm}C{2.7cm} } 
	\topline
		\headcol &  GCDS & NNKCDE & CKDE & FlexCode\\
		\midline
		\multicolumn{5}{c}{Model (M1)}\\
		$\tau=0.05$ & \textbf{0.356}(0.026) & 1.875(0.011) & 1.231(0.015) & 1.327(0.010) \\
		$\tau=0.25$ & \textbf{0.281}(0.018) & 1.355(0.005) & 0.646(0.011) & 0.996(0.012) \\
		$\tau=0.50$ & \textbf{0.263}(0.015) & 1.028(0.009) & 0.386(0.009) & 0.706(0.011) \\
		$\tau=0.75$ & \textbf{0.268}(0.014) & 0.800(0.013) & 0.306(0.008) & 0.469(0.008) \\
		$\tau=0.95$ & \textbf{0.306}(0.015) & 1.106(0.018) & 0.698(0.016) & 0.563(0.018) \\
		\midrule
		\multicolumn{5}{c}{Model (M2)}\\
		$\tau=0.05$ & \textbf{1.044}(0.057) &  5.318(0.028) & 3.797(0.047) & 4.149(0.143) \\
		$\tau=0.25$ & \textbf{0.427}(0.024) &  5.002(0.023) & 2.315(0.051) & 3.405(0.052) \\
		$\tau=0.50$ &  \textbf{0.334}(0.019) &  4.818(0.035) & 1.904(0.062) & 3.086(0.049) \\
		$\tau=0.75$ &  \textbf{0.411}(0.025) & 4.600(0.046) & 2.034(0.065) & 2.635(0.097) \\
		$\tau=0.95$ &  \textbf{0.993}(0.037) & 5.395(0.104) & 4.391(0.077) & 3.523(0.149) \\
		\midrule
		\multicolumn{5}{c}{Model (M3)}\\
		$\tau=0.05$ & 0.994(0.091) & 1.226(0.083) & 28.51(0.729) & \textbf{0.503}(0.096) \\
		$\tau=0.25$ & 1.209(0.117) & 0.525(0.026) & 32.50(1.371) & \textbf{0.270}(0.005) \\
		$\tau=0.50$ & 8.298(0.631) & \textbf{4.125}(0.134) & 57.24(1.578) & 4.551(0.164) \\
		$\tau=0.75$ & \textbf{9.364}(0.806) & 13.52(0.205) & 78.60(2.065) & 19.05(0.708) \\
		$\tau=0.95$ & \textbf{17.31}(0.719) & 58.25(1.495) & 153.3(3.865) & 73.26(2.877) \\
		\midrule
		\multicolumn{5}{c}{Model (M4)}\\
		$\tau=0.05$ & \textbf{0.080}(0.016) & 0.319(0.004) & 0.235(0.009) & 0.082(0.008) \\
		$\tau=0.25$ & \textbf{0.037}(0.006) & 0.150(0.002) & 0.165(0.009) & 0.063(0.005) \\
		$\tau=0.50$ & 0.881(0.048) & 0.314(0.002) & 0.466(0.010) & \textbf{0.406}(0.010) \\
		$\tau=0.75$ & \textbf{0.026}(0.005) & 0.156(0.002) & 0.171(0.008) & 0.070(0.002) \\
		$\tau=0.95$ & \textbf{0.072}(0.018) & 0.313(0.004) & 0.217(0.004) & 0.086(0.007) \\			
		\bottomline
	\end{tabular}
	\caption{Mean squared prediction error (MSE) of conditional quantiles in models (M1) to (M4) and the corresponding simulation standard error(in parentheses). The smallest MSEs are in bold font.}
	\label{tab:conditionalquantile1}
\end{table}

\subsubsection{The two-dimensional helix model}
We now visualize some simulation results of GCDS and compare with NNKCDE and CKDE for a two-dimensional $Y$. FlexCode does not support the case when the dimension of $Y$ is greater than 1, so we do not include it here.

\begin{figure}[H]
	\centering
	\includegraphics[width=4.0in, height=2.8in]{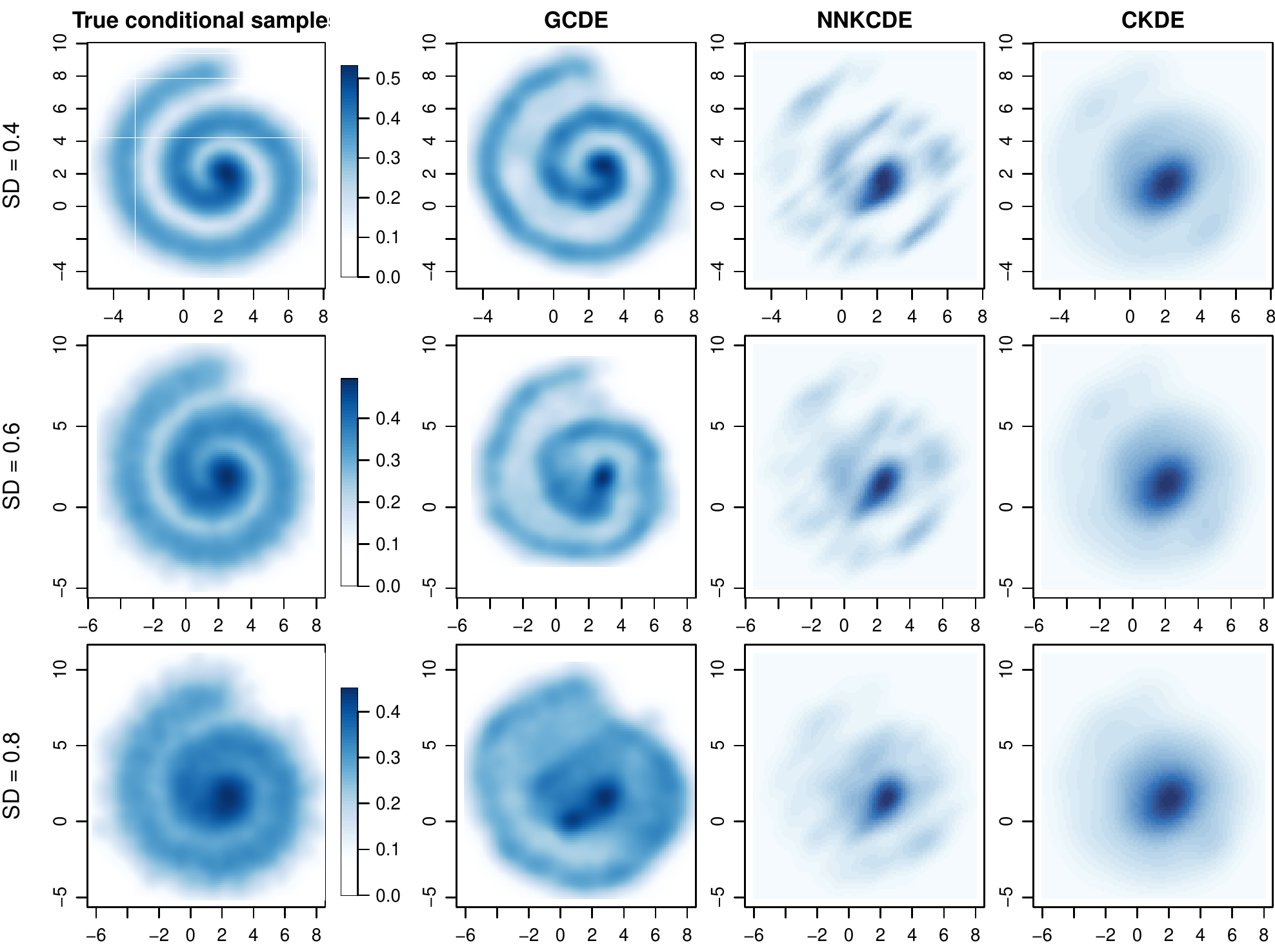}
	\caption{Comparison of conditional density estimation in the 2D conditional helix model. From top to bottom, each row exhibits the conditional density estimates at $X = 1$ with the noise standard deviation $\sigma=0.4, 0.6$, and $0.8$, respectively. }
	\label{fig:2dcondiinvo}
\end{figure}

We consider the conditional helix model:
$$Y_1 = 2X + U\sin(2U) + \varepsilon_1,$$
$$Y_2 = 2X + U\cos(2U) + \varepsilon_2,$$
where $X, U, \varepsilon_1$ and $\varepsilon_2$ are independent, $ X\sim N(0,1), U \sim \text{Uniform}[0,2\pi], \varepsilon_1\sim N(0,\sigma^2)$ and $\varepsilon_2\sim N(0,\sigma^2).$
In this model,  $(U, \varepsilon_1, \varepsilon_2)$ contributes to the noise terms and  is not observed. The value of $\sigma$ determines the noise level. The conditional distribution is a helix with random noise.

We visualize the quality of conditional samples and conditional density estimation given $X = 1$ at three noise levels, $\sigma = 0.4, 0.6$, and $0.8$.
The results are shown in Figure \ref{fig:2dcondiinvo}.
We see clearly that GCDS yields the best estimation of the conditional densities.

\subsection{Effects of the neural network structure}

\subsubsection{Simulated data}
We consider three neural network models for the generator $G$ and another three corresponding neural network models for the discriminator $D$. Each set of the three networks include
 \textit{Half}, \textit{Base} and \textit{Double}, where the size of \textit{Half} network is half  of
 the \textit{Base}, and the size of  \textit{Double} is about double of \textit{Base}.
 The detailed specifications of these networks are given in Table \ref{netspecSR}.
 The \textit{Base} network is used in the simulation studies reported in the manuscript.
The number of replications in each setting is 10. The values of $m$ are given as described above, that is, for the simulation models (M1) to (M3), we set $m=3$; for (M4), $m=4$.

\begin{table}[H]
	\centering
\label{netspecSR}
{\scriptsize
		\begin{tabular}{c|ll|ll|ll}\hline
& \multicolumn{2}{c|}{\textit{\normalsize Half}}   & \multicolumn{2}{c|}{\textit{\normalsize Base}}  & \multicolumn{2}{c}{\textit{\normalsize Double}} \\ \hline
 Model   & $G$         & $D$
              &$ G$         & $D$
              & $G$         & $D$          \\  \hline
 (M1)-(M3)  & $L=1$ & $L=2$
           & $L=1$ & $L=2$
           & $L=2$ & $L=3$ \\
           & $ W=(25)$ & $ W=\{25, 13\}$
           & $W=(50)$ & $W=(50, 25) $
           & $W=(100, 50)$ & $ W=(100, 50, 25)$ \\ \hline
 (M4)  & $L=2$ & $L=2$
           & $L=2$ & $L=2$
           & $L=3$ & $L=3$ \\
           & $ W=(20,8)$ & $ W=\{25, 13\}$
           & $W=(40,15)$ & $W=(50, 25) $
           & $W=(80,30,30)$ & $ W=(100, 50, 25)$ \\ \hline
\end{tabular}
\caption{Simulated data from models (M1)-(M4): network parameter specifications. $L$: the number of hidden layers, $W$: the widths of the layers.}
}
\end{table}

Table \ref{tab:1d_wd} shows the results. Since model (M3) has a larger range of distribution and heavy tail, it is more sensitive to the choice of width and depth. However, this may be relieved by some simple transformation(For example, take logarithm). For other models, there will be some increase in MSE if the width and depth is not optimal, but the difference is mild.

\begin{table}[!ht]
	\centering
	\begin{tabular}{c|rrr}

		\hline
		& \textit{Half \ \ \ \ } & \textit{Base\ \ \ \ } & \textit{Double \ \ \   } \\ \hline
		(M1) & 0.187(0.007) & 0.152(0.005) & 0.240(0.014) \\
		\hline
		(M2) & 0.410(0.025) & 0.333(0.024) & 0.717(0.058) \\
		\hline
		(M3)  & 11.583(8.515) & 4.031(0.425) & 120.046(7.585) \\
		\hline
		(M4) & 0.018(0.005) & 0.019(0.003) & 0.042(0.008) \\
		\hline
	\end{tabular}
	\caption{MSE for estimating the conditional mean in the models (M1)-(M4) by GCDS, using
the three neural network models \textit{Half}, \textit{Base} and \textit{Double}.}
	\label{tab:1d_wd}
\end{table}

\subsubsection{MNIST data: effects of network structures and noise dimension $m$}

To study the influence of the network parameters (width and depth) on the quality of the generated images with the MNIST dataset, we compare three models: \textit{Small}, \textit{Median} and
\textit{Large}. The network parameters (depth, width) are given in Table \ref{NetSrb}.

\begin{table}[H]
	\centering
\label{NetSrb}
{\scriptsize
		\begin{tabular}{ll|ll|ll} \hline
\multicolumn{2}{c|}{\textit{\normalsize Small}}   & \multicolumn{2}{c|}{\textit{\normalsize Median}}  & \multicolumn{2}{c}{\textit{\normalsize Large}} \\ \hline
  $G$         & $D$
              &$ G$         & $D$
              & $G$         & $D$          \\  \hline
  $L=2$ & $L=2$
           & $L=3$ & $L=3$
           & $L=4$ & $L=4$ \\
            $ W=(128,256)$ & $ W=(256,128,1)$
           & $W=(128,256,256)$ & $W=(256,128,1)$
           & $W=(128,128,256,256)$ & $ W=(256,128,128,1)$ \\ \hline
\end{tabular}
\caption{MNIST dataset: network parameter specifications. $L$: the number of hidden layers, $W$: the widths of the layers.}
}
\end{table}

\begin{figure}[H]
	\centering
	\includegraphics[width=6 in]{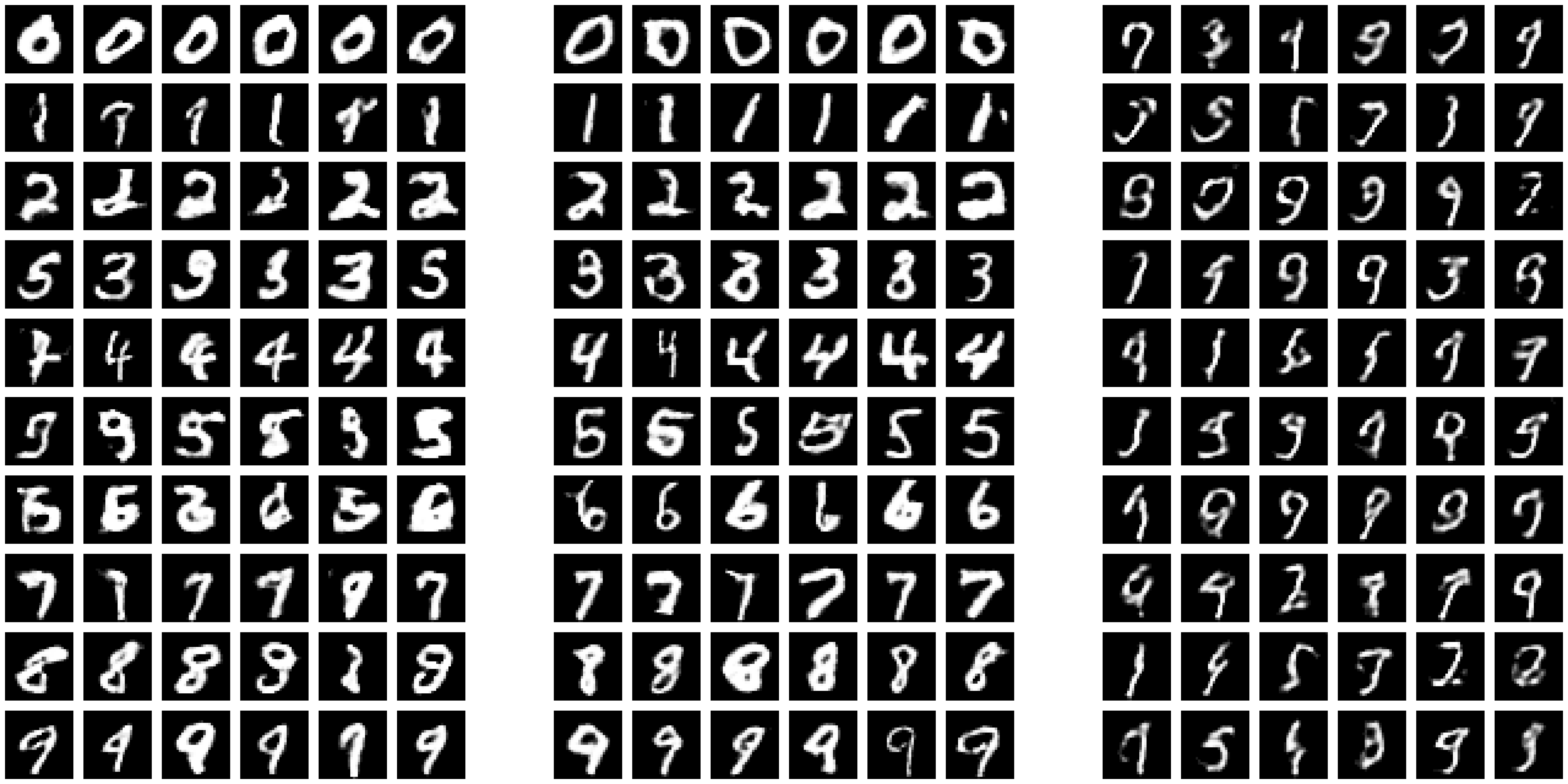}
	\caption{Generated images using the \textit{Half}, \textit{Base} and \textit{Double} neural networks. Left panel: \textit{Half}, middle panel: \textit{Base}, right panel: \textit{Double}.}
	\label{fig:mnist_wd}
\end{figure}

Figure \ref{fig:mnist_wd} shows three panels of images generated by GCDS using \textit{Small, Median} and
\textit{Large} models from left to right. A small network has low capacity to capture the image distribution so the quality of the images are poor. Too large a network has convergence issues so the right panel does not show meaningful images.

We also examine the effects of three different values of  the noise dimension $m$.
We consider $m=$ 10, 100, 200 in GCDS for generating digital images with the MNIST dataset.
\begin{figure}[H]
	\centering
	\includegraphics[width=6 in]{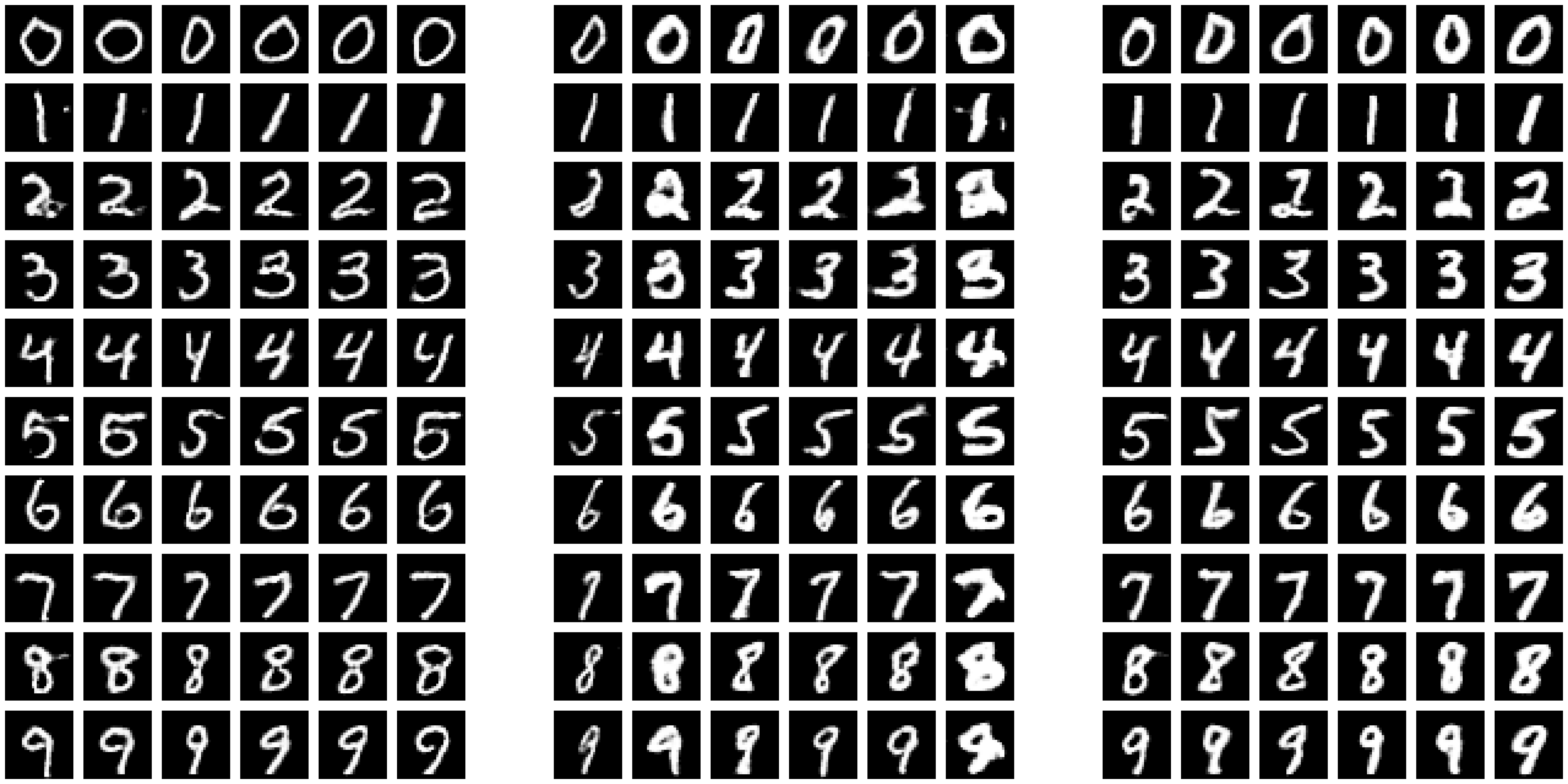}
	\caption{Generated images from GCDS with $m =$ 10 (left panel),
$m=$100 (middle panel), and $m=$200 (right panel)}
	\label{fig:mnist_m}
\end{figure}

Figure \ref{fig:mnist_m} shows the  generated images from GCDS  trained with  three different values of $m$. From left to right, the panels correspond to the $m=$10,100 and 200, respectively. The quality of the all the images look reasonable.
However, when $m=10$, there is very little variation among the images for a given
digit. This means that the learned conditional distribution of the image given the label
is essentially degenerate.
The generated images with $m=100$ and 200 are similar. They also show reasonable variations in the generated images.
This example suggests that for the complex task of image generation with high-dimensional image data, the value of $m$ should not be too small, and as long as $m$ takes a reasonable large value
(e.g., no smaller than the intrinsic data dimension), the results are stable.
Of course, for different types of image data, the value of $m$ should be determined on a case-by-case basis.

\subsection{Performance of GCDS as sample size changes}
We conduct simulation studies to demonstrate that
the performance of GCDS improves as the sample increases, which supports the consistency result we obtained and also indicate reasonable convergence properties of GCDS.

In the first set of experiments, we consider models (M1)-(M4).  We repeat our method with different sample sizes $n=1000, 2500, 5000$ and $7500$.
For every model, we run simulations 10 times.  We  record how the mean squared errors of the estimated conditional mean $E(Y|X)$ changes with sample size.
The results are shown in Table \ref{tab:1d_n}.
There is a clear trend that the MSEs  decrease with $n$.

\begin{table}[!htbp]
	\centering
	\begin{tabular}{c|r|r|r|r}
		\hline
		n & 1000 & 2500 & 5000 & 7500  \\
		\hline
		M1 & 0.432(0.023) & 0.262(0.012) & 0.240(0.015) & 0.234(0.013)\\
		\hline
		M2 & 2.148(0.204) & 0.549(0.058) & 0.308(0.024) & 0.281(0.015) \\
		\hline
		M3& 126.082(10.394) & 23.496(12.155) & 3.219(0.207) & 2.339(0.136) \\
		\hline
		M4 & 0.060(0.008) & 0.036(0.007) & 0.018(0.002) & 0.016(0.003)\\
		\hline
	\end{tabular}
	\caption{MSE of the estimated condition expectation using GCDS  trained with  different sample sizes (simulation standard errors in parentheses).}
	\label{tab:1d_n}
\end{table}

In the second set of experiments,  we examine how the samples size affects the generated images with the MNIST dataset. In this dataset, the total training set size is 60,000.
We train GCDS
with sample sizes $n=$5,000, 10,000 and 60,000.

\begin{figure}[H]
	\centering
	\includegraphics[width=6 in, height=2.5 in]{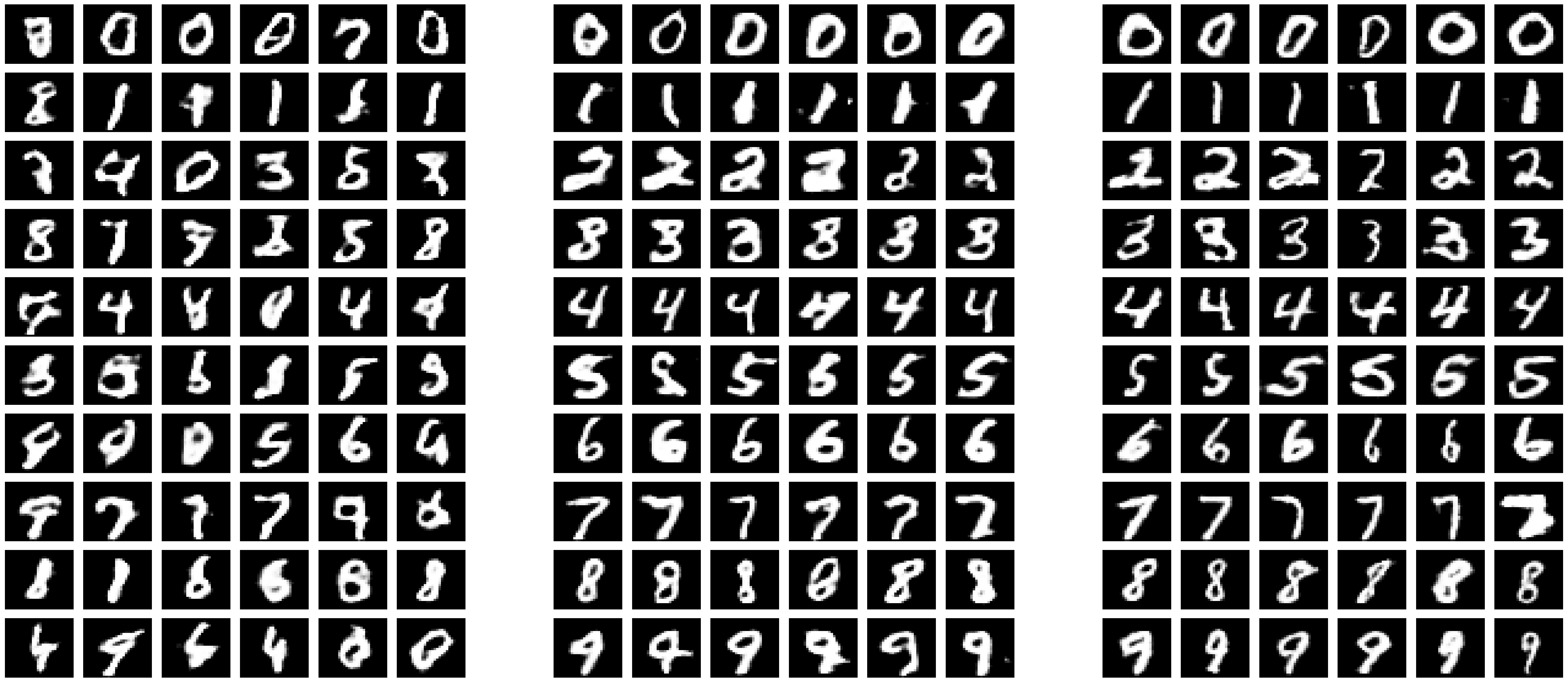}
	\caption{From the left panel to the right panel, generated images using GCDS trained with sample sizes $n=$ 5,000, 10,000, and 60,000.}
	\label{fig:mnist_n}
\end{figure}

Figure \ref{fig:mnist_n} shows the generated images. From left to right, the panels represent the generated images using GCDS trained with  $n=$5,000, 10,000 and 60,000 images in the MNIST dataset.
The rows consist of generated images conditioning on the labels from 0 to 9. In the left panel consisting of images with sample size 5,000, we see that the quality of the images is poor. The quality of the images improves as the sample size increases to 10,000 and then 60,000.


\subsection{Image reconstruction with  
STL-10 dataset}

\begin{figure}[H] 
	\centering
	\includegraphics[width=5.6in, height=3.6in]{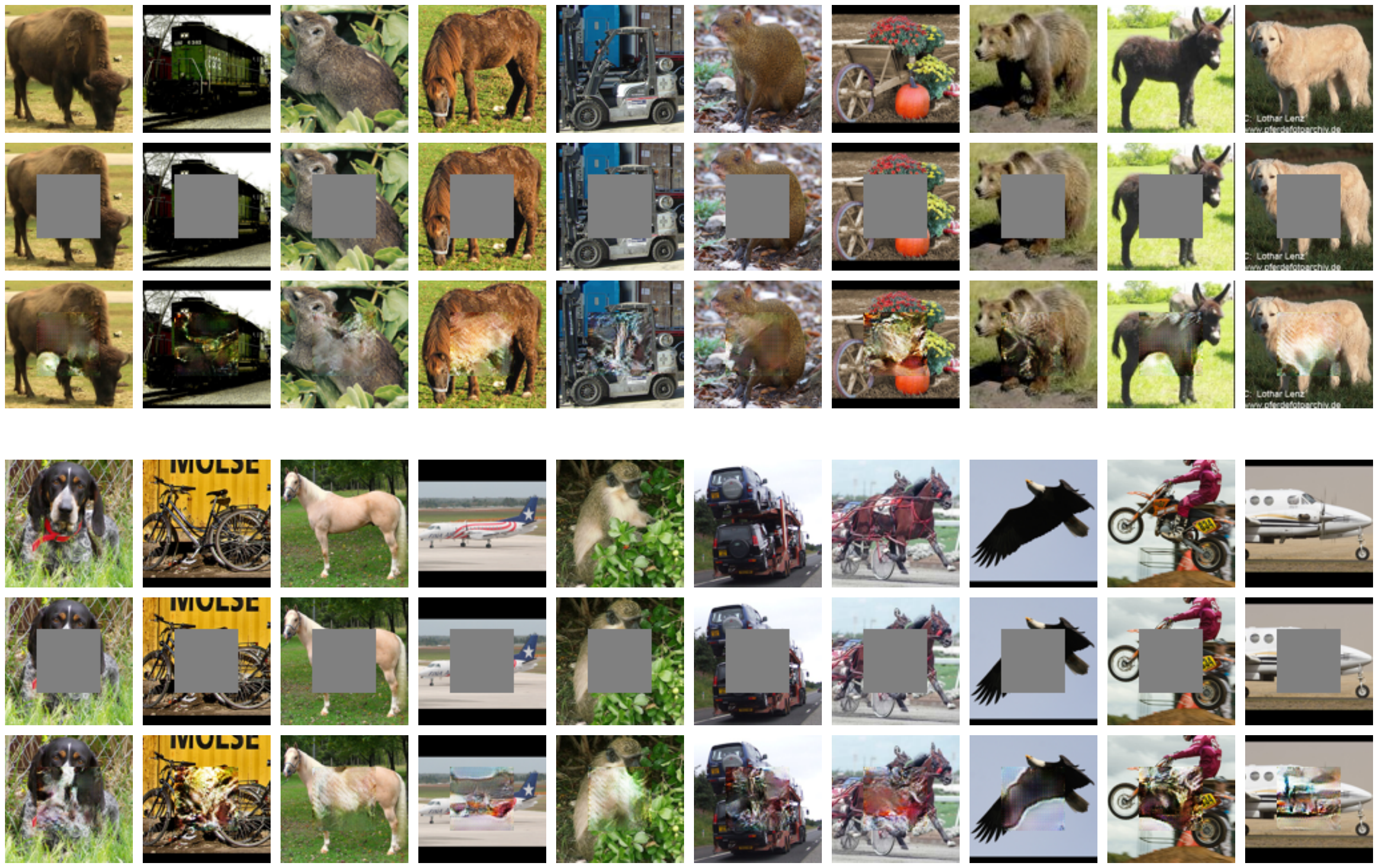}
	\caption{Reconstructed images given partial image in STL10 dataset. In each panel, the first row consists of the true images, the second row consists of images with the corrupted area. The third row contains the reconstructed images.}
	\label{fig:stl10_ximage}
\end{figure}

As a proof of concept, we illustrate GCDS with a different high-dimensional data problem. We use the STL-10 \citep{coates11a}
dataset, which contains  color images of size $96\times 96 \times 3$.
We apply GCDS to the image reconstruction problem with the STL-10 dataset,  where the central part of  the image  is deliberately corrupted. The size of the corrupted area is a quarter of the whole image.  The task is to reconstruct the corrupted area.
This is also known as image inpainting.
We formulate this task as a problem  of generating from a high-dimensional conditional distribution. With this formulation, the observed part of the image is the given information while we would like to generate the missing area given the observed part of the image. For the STL-10 dataset, the response $Y$ is the corrupted part of the image, whose dimension is $48 \times 48 \times 3=6,912$. The  predictor $X$ is the observed part of the image, whose dimension is $96\times 96\times 3 -48 \times 48 \times 3=20,736.$
We adopted a neural network with a structure similar to that of \citet{iizuka2017globally}.
In each panel of Figure \ref{fig:stl10_ximage}, the first row consists of the true images, the second row shows the images with corrupted area, and the third row contains the reconstructed images.
We can see that the reconstructed images are of good quality comparing with the original images.
This suggests that GCDS is a promising framework for image reconstruction tasks. However, more work is needed to apply GCDS to more complex image reconstruction problems, for example, if there are multiple corrupted areas in an image. This is beyond the scope of the present work and will be studied in the future.


\section{Proofs and additional technical details}
In this section, we give the proofs of the results stated in the paper and present additional details about $f$-divergence.

\subsection{Proof of Lemma \ref{lem2}}
\begin{proof}
	Our proof follows \citet{keziou2003dual}.
	Since $f(t)$ is convex, then for $t \in \mathbb{R}$, we have  $f(t)=f^{**}(t)$, where $$f^{**}(t)=\sup_{s\in\mathbb{R}} \{st - f^*(s)\}$$ is the Fenchel conjugate of $f^*$. By  Fermat's rule, the maximizer  $s^*$  satisfies $$t\in \partial 
f^*(s^*),$$ i.e., $$s^*\in \partial f(t)$$
	Plugging  the above display with $t = \frac{\mathrm{d} \mu_Z}{\mathrm{d} \gamma}(x)$ into the definition of $f$-divergence,  we obtain (\ref{fdual}). This completes the proof.
\end{proof}

\subsection{Proof of Lemma \ref{NOLemma}}
\begin{proof}
First, based on the basic noise-outsourcing lemma (Theorem 5.10 in
\cite{kall2002},  Lemma 3.1 in \cite{austin2015}), there is a uniform random variable
$u \sim  \text{Uniform}\,[0, 1]$ and a measurable function
$G_1: [0, 1] \times \cX \to \cY$ such that
\[
Y= G_1(u, X) \ \text{almost surely}.
\]
For a random vector $\eta \sim N(\0, \bI_m)$  for some $m \ge 1$, there exists
a  measurable function $G_2: \mathbb{R}^m \to [0, 1]$ such that
\[
u = G_2(\eta) \ \text{ almost surely}.
\]
For example, we can simply take $G_2(\eta) = \Phi^{-1}(\eta_1)$, where $\Phi^{-1}$ is the inverse
of the standard normal distribution function and $\eta_1$ is the first component of $\eta$. It is well-known that $\Phi^{-1}(\eta_1) \sim \text{Uniform}\, [0, 1].$
Combining the above two equations, we have
\[
Y = G_1(G_2(\eta), X) \text{ almost surely}.
\]
So we can simply take $G$ as
\[
G(\eta, x) = G_1(G_2(\eta), x), \ (\eta, x) \in \mathbb{R}^m\times \cX.
\]
Therefore, we have
\[
(X, G(\eta, X) = (X, Y) \ \text{almost surely}.
\]
This completes the proof.
\end{proof}


\subsection{$f$-divergence,  convex dual  and variational form}

First, we note that the definition of $f$-divergence used in this work differs from that in
\cite{nguyen2010estimating}.
The definition of $f$-divergence we used as defined in (\ref{fdiv}) is
\begin{equation}
\label{fdR}
\Dbb_f(q\Vert p) = \int p(z) f\left(\frac{q(z)}{p(z)}\right) d z.
\end{equation}
The definition in \cite{nguyen2010estimating} is
\begin{equation}
\label{phidR}
\Dbb_{\phi}(q \Vert p) = \int q(z) \phi \left(\frac{p(z)}{q(z)}\right) d z,
\end{equation}
where $\phi: \mathbb{R}\to \mathbb{R}$ is  a convex and lower-semicontinuous function.
Since
\[
\Dbb_f(q\Vert p) = \int p(z) f\left(\frac{q(z)}{p(z)}\right) d z
= \int q(z) \frac{p(z)}{q(z)} f \left(\frac{q(z)}{p(z)}\right) d z,
\]
the two definitions (\ref{fdR}) and (\ref{phidR}) are equivalent if we set $f (x) = x \phi(1/x)$.
For example, for the KL divergence, in definition (\ref{fdR}), we take $f(x)=x \log x$, then
\[
\Dbb_{f}(q\Vert p)  = \int p(z) \frac{q(z)}{p(z)} \log \left( \frac{q(z)}{p(z)}\right) dz
 = \int q(z) \log \left(\frac{q(z)}{p(z)}\right) dz =\Dbb_{\text{KL}}(q\Vert p).
\]
In definition (\ref{phidR}), we take $\phi(x)= \log(1/x)$, which yields the same expression as above, i.e.,
\[
\Dbb_{\phi}(q\Vert p)  = \int q(z) \log \left( \frac{q(z)}{p(z)}\right) dz =\Dbb_{\text{KL}}(q\Vert p).
\]
For these two functions $f(x)=x\log x$ and $\phi(x)=\log(1/x)$, we have
$f(x)=x \phi(1/x)$.

We now give the expressions of the $f$-functions, their dual forms $f^*$ and the variational forms of the KL-, JS- and $\chi^2$-divergences. For detailed derivations, we refer to
\cite{rockafellar1970convex}, \cite{keziou2003dual}, and \cite{nguyen2010estimating}.

\begin{itemize}
\item
KL-divergence:
\begin{eqnarray*}
f(x)&=&x \log x,\
f^*(t)= e^{t-1}, \\
D_{\text{KL}}(q\Vert p)
&=& \sup_D \{\Ebb_{Z\sim q} D(Z) -\Ebb_{W\sim p} e^{D(W)-1}\}\\
&=&\sup_D\{\Ebb_{Z\sim q} D(Z) -\Ebb_{W\sim p} e^{D(W)}\}+1.
\end{eqnarray*}

\item JS-divergence:
\begin{eqnarray*}
f(x) &=& -(x+1) \log \frac{1+x}{2}+x \log x,\
f^*(t)=-\log(2-\exp(t)),\\
D_{\text{JS}}(q\Vert p)
&=& \sup_D \{\Ebb_{Z\sim q} D(Z) +\Ebb_{W\sim p} \log(2-\exp(D(W)))\}\\
&=&\sup_D \{\Ebb_{Z\sim q} \log D(Z) +\Ebb_{W\sim p} \log(1-D(W))\}
+ \log 4.
\end{eqnarray*}
where the last equality is obtained by change of variable:  $D \to \log(2D)$. This is the form (without the constant $\log 4$) used in the GAN objective function \citep{goodfellow14}.

\item $\chi^2$-divergence:
\begin{eqnarray*}
f(x)&=&(x-1)^2,\
f^*(t)= t + \frac{t^2}{4},  \\
D_{\chi^2}(q\Vert p)
&=& \sup_D \{\Ebb_{Z\sim q} D(Z) -\Ebb_{W\sim p} [D(W)+\frac{D^2(W)}{4}]\}\\
&=&\sup_D\{2\Ebb_{Z\sim q} D(Z)-\Ebb_{W\sim p} D^2(W)\}-1,
\end{eqnarray*}
where the last equality is obtained by change of variable $D \to 2 (D-1)$.
\end{itemize}

\subsection{Proof of Theorem \ref{consistency}}

For ease of reference, we first restate the notation and Theorem \ref{consistency} before giving the proof.

Recall ${\cL}(G,D)$ is defined as
\begin{eqnarray}\label{lossps}
\cL(G,D) = 
\Ebb_{(X, \eta)\sim P_X P_{\eta}} D(X, G(\eta,X))
- \Ebb_{(X, Y) \sim P_{X,Y}} \exp(D(X,Y)).
\end{eqnarray}
For any measurable function $G: \real^m \times \real^d \mapsto \real^q$,  define
\begin{equation}\label{lossklps}
\mathbb{L}(G) 
=  \sup_{D} \cL(G,D).
\end{equation}
Let
\begin{equation}
\label{dualobjs}
\widehat{\cL}(G, D) =
\frac{1}{n}\sum_{i=1}^n D(X_i, G(\eta_i, X_i))
- \frac{1}{n}\sum_{i=1}^n \exp(D(X_i, Y_i)).
\end{equation}
For a fixed $G$, let $p_{XG}$ be the joint density of $(X, G(\eta, X))$. Lemma \ref{lem2} implies that the optimal $D$ is
\[
D^*(z) = \log \frac{p_{XG}(z)}{p_{XY}(z)}=\log r(z).
\]
Thus the optimal discriminator is the log-likelihood ratio serving as a critic of the resemblance between $p_{XY}$ and $p_{XG}$. Substituting this expression into (\ref{lossklps}), we have,
\[
\Lbb(G) = \Ebb_{(X, \eta) \sim P_XP_{\eta}} [\log r(X, G(\eta, X))].
\]
Therefore, the optimal $G^*$ minimizing  the KL-divergence satisfies $P_{(X, G^*(X, \eta))} = P_{X, Y}$ by Lemma \ref{distm}.

We make the following assumptions.

\begin{itemize}
\item[(A1)]

The target conditional generator $G^*: \real^{m} \times \cX  \to \cY$ is continuous with $\|G^*\|_{\infty} \le C_0$ for some constant $0 < C_0 < \infty.$


\item[(A2)]

For any  $G \in \cG\equiv \cG_{\mathcal{D}, \mathcal{W}, \mathcal{S}, \mathcal{B}}$,
$r_G(z) = {p_{X,G(\eta,X)}(z)}/{p_{X,Y}(z)}: \cX\times \cY \to \real$ is continuous and
$0<C_1\leq r_G(z)\leq C_2$ for some constants $0 < C_1 \le C_2 < \infty.$

\end{itemize}

%

We also make the following assumptions on the network parameters
of the conditional generator $G_{\vtheta}$  and the discriminator  $D_{\vphi}$.

\begin{itemize}
\item[(N1)] The network parameters of $\cG$ satisfies
	\begin{equation*}
		\mathcal{H}\mathcal{W}\to\infty \quad {\rm and} \quad \frac{\mathcal{B}\mathcal{S}\cH\log(\mathcal{S})\log n}{n}\to 0,
\text{ as } n \to \infty.
	\end{equation*}

\item[(N2)]
The network parameters of $\cD$ satisfies
	\begin{equation*}
		\tilde{\mathcal{H}}\tilde{\mathcal{W}}\to\infty \quad {\rm and} \quad \frac{\tilde{\mathcal{B}}\tilde{\cS}\tilde{\cH}
\log(\tilde{\cS})\log n}{n}\to 0, \text{ as } n \to \infty.
	\end{equation*}
\end{itemize}


\medskip\noindent
\textbf{Theorem \ref{consistency}}
\textit{
Suppose that the assumptions (A1) and (A2) hold. If the network parameters of $\cG$ and $\cD$ satisfies the specifications (N1) and (N2), then
\[
\Ebb_{(X_i,Y_i,\eta_i,)_{i=1}^n}  \|p_{X, \hat{G}_{\vtheta}(\eta, X)}-p_{X, Y}\|_{L_1}^2 \rightarrow 0,
\ \text{ as } n \to \infty.
\]
}

For simplicity, we only prove Theorem \ref{consistency}
for $q=1$.
By a truncation argument, we only need to  consider on domains  $E_1 = [-B,B]^{d+m}\in \mathbb{R}^{d}\times{R}^{m}$ and  $E_2 = [-B,B]^{d+q}\in \mathbb{R}^{d}\times{R}^{q}$ with $B = \log n$.
Let   $C_3 = \max\{|\log(C_1)|,|\log(C_2)|\}+1$ and
 $C_4 = \max_{|s|\leq  2 C_3}  \exp(s)+2C_3=\exp(2C_3)+2C_3$.
Define the empirical loss function at the sample level by
 \begin{equation}\label{losskls}
\hat{\mathbb{L}}(G) = \hat{\mathbb{D}}(p_{X, G(\eta,X)},p_{X,Y}) =  \sup_{D_{\vphi}\in \cD} \widehat{\cL}(G,D_{\vphi}),
\end{equation}
where $\widehat{\cL}(G,D)$ is  defined in  \eqref{dualobj}.
 Then,
 \[
 \hat{G}_{\vtheta} \in  \arg \min_{G_{\vtheta}\in \cG} 
  \widehat{\cL}(G_{\vtheta},\hat{D}_{\vphi}),
 \]
 and
\[
\hat{D}_{\vphi} \in  \arg \max_{{D}_{\vphi}\in \cD} 
\widehat{\cL}(\hat{G}_{\vtheta}, D_{\vphi}).
\]
To shorten the notation, we use $\hat{G}$, $\hat{D}$, $\mathcal{G}$, and $\mathcal{D}$ to denote $\hat{G}_{\vtheta}$, $\hat{D}_{\vphi}$, $\mathcal{G}_{\mathcal{H}, \mathcal{W}, \mathcal{S}, \mathcal{B}}$, and $\mathcal{D}_{\tilde{\mathcal{H}}, \tilde{\mathcal{W}}, \tilde{\mathcal{S}}, \tilde{\mathcal{B}}}$, respectively.

We first give a high-level description of the proof. By Lemma \ref{KLM}, $\mathbb{L}(G^*) = 0.$
Then, by Pinsker's inequality \citep{npe2008}, we have
\begin{equation}\label{errl1s}
 \|p_{X, \hat{G}(\eta, X)}-p_{X, Y})\|_{L^1}^2 \leq 2 (\mathbb{L}(\hat{G}) - \mathbb{L}(G^*)).
 \end{equation}
Next we show that the right side in \eqref{errl1s} goes to $0$.

For any	$\bar{G}\in \cG$, we decompose the right side in \eqref{errl1s} as follows:
\begin{eqnarray}
\mathbb{L}(\hat{G})-\mathbb{L}(G^{*}) 
& = &\sup_{D} \cL(\hat{G},D)  -  \sup_{D \in \cD} \cL(\hat{G},D) \nonumber\\
& &+   \sup_{D \in \cD} \cL(\hat{G},D) - \sup_{D \in \cD} \widehat{\cL}(\hat{G},D) \label{ep1s}\\
& &+ \sup_{D \in \cD} \widehat{\cL}(\hat{G},D) - \sup_{D \in \cD} \widehat{\cL}(\bar{G},D) \label{neg1s}\\
& &+ \sup_{D \in \cD} \widehat{\cL}(\bar{G},D) - \sup_{D \in \cD} \cL(\bar{G},D) \label{ep2s}\\
& &+\sup_{D \in \cD} \cL(\bar{G},D) - \sup_{D} \cL(\bar{G},D)\label{neg2s}\\
& &+ \sup_{D} \cL(\bar{G},D)-\sup_{D} \cL(G^*,D).\nonumber
\end{eqnarray}
Since the terms in \eqref{neg1s} and \eqref{neg2s} are nonpositive,
and the terms in \eqref{ep1s} and  \eqref{ep2s}  are smaller than
$$\sup_{D\in \cD, G\in \cG}|\cL(G,D)-\widehat{\cL}(G,D)|,$$
we have
\begin{eqnarray*}
\mathbb{L}(\hat{G})-\mathbb{L}(G^{*})
&\leq & \sup_{D} \cL(\hat{G},D)  -  \sup_{D \in \cD} \cL(\hat{G},D) + 2 \sup_{D\in \cD, G\in \cG}|\cL(G,D)-\widehat{\cL}(G,D)|\nonumber\\
& &+ \sup_{D} \cL(\bar{G},D)-\sup_{D} \cL(G^*,D).\nonumber
\end{eqnarray*}
Taking infimum on $\bar{D}$ over $\cD$ on both side of the above display we get
\begin{equation}\label{errdecs}
\mathbb{L}(\hat{G})-\mathbb{L}(G^{*}) \le
\Delta_1 + \Delta_2 + \Delta_3,
\end{equation}
where
\begin{eqnarray*}
\Delta_1&=&\sup_{D} \cL(\hat{G},D)  -  \sup_{D \in \cD} \cL(\hat{G},D),\\
\Delta_2&=&
2 \sup_{G \in \cG, D\in \cD}|\cL(G,D)-\hat{L}(G,D)|,\\
\Delta_3&=&\inf_{\bar{G}\in \cG} [\mathbb{L}(\bar{G})-\mathbb{L}(G^*)].
\end{eqnarray*}
The first and the third terms $\Delta_1$ and $\Delta_3$
are the approximation errors and
the second term $\Delta_2$ is the statistical error.
To prove the theorem, it suffices to show that these error terms converge to zero.

We first show the following Lemmas \ref{app1} - \ref{app2} to bound these two types of errors  in \eqref{errdecs}.

\begin{lemma}\label{app1}
$\Delta_3\equiv \inf_{\bar{G}\in \cG} [\mathbb{L}(\bar{G})-\mathbb{L}(G^*)] = o(1)$, as $n\rightarrow \infty.$
\end{lemma}
\begin{proof}
By the assumption,  $G^*$ is   continuous on $E_1 = [-B,B]^{d+m}$ with $B = \log n$,   and $\|G^*\|_{L^{\infty}} \leq C_0$.
Setting $L = \log n$, $N = n^{\frac{d+m}{2(2+d+m)}}/\log n$, $E = E_1$  and $R=B$,  in  Lemma \ref{appshen},  we get
an  ReLU network $\bar{G}_{\bar{\vtheta}} \in \cG$
 with depth $\mathcal{H} = 12 \log n + 14+2(d+m),$ width $\mathcal{W} = 3^{d+m+3}\max\{(d+m)(n^{\frac{d+m}{2(2+d+m)}}/\log n)^{\frac{1}{d+m}},n^{\frac{d+m}{2(2+d+m)}}/\log n+1\},$ and size $\mathcal{S} =n^{\frac{d+m-2}{d+m+2}}/(\log^4 n),$ $\mathcal{B} = 2C_0$ such that
		\begin{equation}\label{appg}
\|G^*-\bar{G}_{\bar{\vtheta}}\|_{L^{\infty}(E_1)} \leq  19 \sqrt{d+m} \omega_{f}^{E_1}
\left(2 (\log n ) n^{\frac{-1}{2+d+m}} \right).
\end{equation}
Let $\bar{D} = \log \frac{p_{X, \bar{G}_{\bar{\vtheta}}(\eta, X)}(z)}{p_{X,Y}(z)}$ and
  ${D}^* = \log \frac{p_{X, {G}^*(\eta, X)}(z)}{p_{X,Y}(z)}.$ Then,
  the above display on $\|G^*-\bar{G}\|$ in (\ref{appg}) and the continuity  implies that
  $\|{D}^* - \bar{D}\| \rightarrow 0 $ as $n\rightarrow \infty$.
   Therefore, by Lemma \ref{lem2},  we have  $$ \mathbb{L}(\bar{G})=\sup_{D} \cL(\bar{G},D) = \Ebb_{(X, \eta)\sim P_X P_{\eta}} \bar{D}(X, \bar{G}(\eta,X)) - \Ebb_{(X, Y) \sim P_{X,Y}} \exp(\bar{D}(X,Y))$$ converge to
 $$ \mathbb{L}(G^*)=\sup_{D} \cL(G^*,D) = \Ebb_{(X, \eta)\sim P_X P_{\eta}} {D}^*(X, {G}^*(\eta,X)) - \Ebb_{(X, Y) \sim P_{X,Y}} \exp({D}^*(X,Y))$$
 as $n \rightarrow \infty$.
\end{proof}

 \begin{lemma}\label{errsta}
 \begin{equation}\label{sta}
\Delta_2\equiv \sup_{D\in \cD, G\in \cG}|\cL(G,D)-\widehat{\cL}(G,D)| \leq \mathcal{O} (n^{-\frac{2}{2+d+m}}+n^{-\frac{2}{2+d+q}}).
\end{equation}
\end{lemma}
\begin{proof}
For any  $G \in \cG$,
  by assumption,  $D_{G}(z) = \log r_G(z) = \log (\frac{p_{X,G(\eta,X)}(z)}{p_{X,Y}(z)})$ is   continuous on $E_2 = [-B,B]^{d+q}$ with $B = \log n$ and $\|D_{G}\|_{L^{\infty}} \leq C_3$.
Setting $L = \log n$, $N = n^{\frac{d+q}{2(2+d+q)}}/\log n$, $E = E_2$  and $R=B$,  in  Lemma \ref{appshen},  we get
an  ReLU network $\bar{D}_{\bar{\vphi}} \in \cD$
 with depth $\tilde{\mathcal{H}} = 12 \log n + 14+2(d+q),$ width $\tilde{\mathcal{W}} = 3^{d+q+3}\max\{(d+q)(n^{\frac{d+q}{2(2+d+q)}}/\log n)^{\frac{1}{d+q}},n^{\frac{d+q}{2(2+d+q)}}/\log n+1\},$ and size $\tilde{\mathcal{S}} =n^{\frac{d+q-2}{d+q+2}}/(\log^4 n),$ $\tilde{\mathcal{B}} = 2C_3$ such that
	\begin{equation}\label{appd}
\|D_G-\bar{D}_{\bar{\vphi}}\|_{L^{\infty}(E_2)} \leq 19 \sqrt{d+m} \omega_{f}^{E_2}
(2(\log n) n^{\frac{-1}{2+d+q}}).
\end{equation}
Let $Z = (X,Y) \sim P_{X,Y}$ and $Z_i =(X_i,Y_i), i = 1,..,n$ are i.i.d copies  of $Z$. Let $\eta \sim P_{\eta}$ and $\eta \indep X$,   $\eta_j, j= 1,...n$ are i.i.d copies of $\eta$. 
Then $W_{i} =(X_i,\eta_i)$
are i.i.d copies of $W  = (X,\eta)\sim P_{X}P_{\eta}.$
Let $S = (W,Z) \sim (P_{X}P_{\eta}) \bigotimes P_{X,Y}$ and
let $S_{i} = (W_i,Z_i) = ((X_i,\eta_i),(X_i,Y_i)), i=1,\ldots, n$
be $n$ i.i.d copies of $S$. Denote $$b(G,D;S) =  D(X, G(\eta,X)) -   \exp(D(X,Y)).$$
		Then
		$$\mathcal{L}(G,D) = \mathbb{E}_{S} [b(G,D;S)]$$
		and
		$$ \widehat{\mathcal{L}}(G,D) = \frac{1}{n}\sum_{i=1}^n b(G,D;S_i).$$
Let $\epsilon_i, i =1,\ldots, n$ be i.i.d Rademacher random samples that are independent of $S_i, i = 1,\ldots,n.$
Denote the Rademacher  complexity of $ \cD\times \cG $  \citep{bartlett2002rademacher} by
		$$\mathcal{C}(\cD\times\cG) = \frac{1}{n} \mathbb{E}_{\{S_i, \epsilon_i\}_{i=1}^n}
\Big[\sup_{G \in \cG, D \in \cD}\Big|\sum_{i=1}^n\epsilon_i b(G,D;S_i)\Big|\Big].
$$

Let  $\mathfrak{C}(\cD\times \cG, e_{n,1}, \delta))$ be  the covering number of $ \cD\times \cG $ with respect to the  empirical  distance
$$e_{n,1}((G,D),(\tilde{G},\tilde{D})) = \frac{1}{n} \mathbb{E}_{\epsilon_i}\Big[\sum_{i =1}^n \Big|\epsilon_i (b(G,D;S_i) - b(\tilde{G},\tilde{D};S_i))\Big|\Big].$$

First, by the standard symmetrization technique and the law of iterated expectations, we have
\begin{eqnarray}
\label{eqn1}
\lefteqn{\sup_{D\in \cD, G\in \cG}|\cL(G,D)-\widehat{\cL}(G,D)|} \nonumber \\
 &=& 2\mathcal{C}(\cD\times\cG) \nonumber \\
&=& 2\mathbb{E}_{S_1,\ldots,S_n}
\big\{\mathbb{E}_{\epsilon_i,i=1,\ldots,n}
\big[\mathcal{C}(\cG\times\cD)|(S_1,...,S_n)\big]\big\}.
\end{eqnarray}

For $\delta>0$, let $\cD_{\delta}\times\cG_{\delta}$ be  such a covering set at scale $\delta$ of $ \cD\times \cG$.
Then, by the triangle inequality and Lemma \ref{2.9} below, we have
\begin{eqnarray}
\label{eqn2}
\lefteqn{\mathbb{E}_{S_1,\ldots,S_n}
\Big\{\mathbb{E}_{\epsilon_i,i=1,\ldots,n}[\mathcal{C}(\cG\times\cD)|
(S_1,...,S_n)]\big\}}  \\
&\leq &  \delta +  \frac{1}{n}\mathbb{E}_{S_1,\ldots,S_n}
\Big\{\mathbb{E}_{\epsilon_i,i=1,\ldots,n}
\Big[\sup_{(G,D)\in \cD_{\delta}\times\cG_{\delta}}\Big |\sum_{i=1}^n\epsilon_i b(G,D;S_i)\Big| \, \Big| (S_1,\ldots,S_n)\Big] \Big\}\nonumber \\
&\leq & 2\delta+C_{5}\frac{1}{n}\mathbb{E}_{S_1,\ldots,S_n}\Big\{\Big[\log \mathfrak{C}(\cD\times\cG, e_{n,1}, \delta)\Big]^{1/2} \max_{(G,D) \in \cD_{\delta}\times \cG_{\delta}}
\Big[\sum_{i = 1}^n  b^2(G,D;S_i)\Big]^{1/2}\Big\}. \nonumber
\end{eqnarray}
Since $\|b(G,D;S)\|_{{\infty}} \leq C_4,$ we have
\[
\Big[\sum_{i = 1}^n  b^2(G,D;S_i)\Big]^{1/2}
\le \sqrt{n}C_4.
\]
Therefore,
\begin{eqnarray}
\label{eqn3}
\lefteqn{\frac{1}{n}\mathbb{E}_{S_1,\ldots,S_n}\Big\{(\log \mathfrak{C}(\cD\times\cG, e_{n,1}, \delta))^{1/2} \max_{(G,D) \in \cD_{\delta}\times \cG_{\delta}}\Big [\sum_{i = 1}^n  b^2(G,D;S_i)\Big]^{1/2}\Big\}} \nonumber \\
		&\leq & \frac{1}{n} \mathbb{E}_{S_1,\ldots,S_n}\Big[(\log \mathfrak{C}(\cD\times\cG, e_{n,1}, \delta))^{1/2}\sqrt{n}C_4\Big] \nonumber \\
		&\leq& \frac{C_4}{\sqrt{n}}\Big[\log \mathfrak{C}(\cD, e_{n,1}, \delta)+\log \mathfrak{C}(\cG, e_{n,1}, \delta)\Big]^{1/2}.
\end{eqnarray}
Now since
$\mathfrak{C}(\cG,e_{n,1},\delta)\leq \mathfrak{C}(\cG,e_{n,\infty},\delta)$ (similar result for $\cD$) and
		$$\log \mathfrak{C}(\cG,e_{n,\infty},\delta)) \leq \mathrm{Pdim}_{\cG} \log \frac{2e\mathcal{B}n}{\delta\mathrm{Pdim}_{\cG}},$$
 where
$\mathrm{Pdim}_{\cG}$ is the Pseudo dimension  of $\cG_{\mathcal{H}, \mathcal{W}, \mathcal{S}, \mathcal{B}}$, which satisfies \citep{bartlett2019}
$$C_6 \mathcal{H}\mathcal{S}\log \mathcal{S}\leq \mathrm{Pdim}_{\cG} \leq C_7 \mathcal{H}\mathcal{S}\log \mathcal{S}.$$
Then, we  have
\begin{eqnarray}
\label{eqn4}
\lefteqn{\frac{1}{\sqrt{n}}\Big[\log \mathfrak{C}(\cD, e_{n,1}, \delta)+\log \mathfrak{C}(\cG, e_{n,1}, \delta)\Big]^{1/2}} \nonumber \\
&\le & \frac{1}{\sqrt{n}}\Big[\mathcal{H}\mathcal{S}\log \mathcal{S} \log \frac{\mathcal{B}n}{\delta \mathcal{H}\mathcal{S}\log \mathcal{S}} + \tilde{\mathcal{H}}\tilde{\mathcal{S}}\log \tilde{\mathcal{S}} \log \frac{\tilde{\mathcal{B}}n}{\delta \tilde{\mathcal{H}}\tilde{\mathcal{S}}\log \tilde{\mathcal{S}}} \Big]^{1/2}.
\end{eqnarray}
Then \eqref{sta} follows from (\ref{eqn1}) to (\ref{eqn4}) with  the selection of the network parameters of $ \cD_{\tilde{\mathcal{H}}, \tilde{\mathcal{W}}, \tilde{\mathcal{S}}, \tilde{\mathcal{B}}}, \cG_{\mathcal{H}, \mathcal{W}, \mathcal{S}, \mathcal{B}}$  and  with $\delta = \frac{1}{n}$.
\end{proof}

\begin{lemma}\label{app2}
$
\mathbb{E}_{(X_i,Y_i,\eta_i,)_{i=1}^n}[\Delta_1]\equiv \mathbb{E}_{(X_i,Y_i,\eta_i,)_{i=1}^n} [\sup_{D} \cL(\hat{G},D)  -  \sup_{D \in \cD} \cL(\hat{G},D)] \rightarrow 0 $ as  $n\rightarrow \infty.$
\end{lemma}
\begin{proof}
 Conditioning on the data $(X_i,Y_i,\eta_i,)_{i=1}^n$, the supremum of $\sup_{D} \cL(\hat{G},D) $ is attained at
$D_{\hat{G}}(z) = \log r_{\hat{G}}(z)$ with $$
r_{\hat{G}}(z) = \frac{p_{X,\hat{G}(\eta,X)}(z)}{p_{X,Y}(z)}.
$$
By assumption, $D_{\hat{G}}(z)$ is   continuous on $[-B,B]^{d+q}$  and $\|D_{\hat{G}}\|_{L^{\infty}} \leq C_3$.
Then similar to the proof of (\ref{appd}),
 in  Lemma \ref{errsta}, there exist $\hat{D}_{\phi} \in \mathcal{D}$ such that
 $
\|D_{\hat{G}}-\hat{D}_{\phi}\|_{L^{\infty}(E_2)} \rightarrow 0,
$ as $n\rightarrow \infty.$
Then $$0<\sup_{D} \cL(\hat{G},D)  -  \sup_{D \in \cD} \cL(\hat{G},D)\leq  \cL(\hat{G},D_{\hat{G}})  -   \cL(\hat{G},\hat{D}_{\phi})\rightarrow 0,$$
by continuity.
\end{proof}

\medskip\noindent
\textbf{Proof of Theorem \ref{consistency}}
\begin{proof} The theorem follows from (\ref{errl1s}) ,  (\ref{errdecs}) and Lemmas \ref{app1} - \ref{app2}.
\end{proof}

\medskip\noindent
\textbf{Proof of Corollary \ref{convergb}}
\begin{proof} Let
\[
\Delta(P_{X,\hat{G}_{\vtheta}}, P_{X, Y})
= \Ebb_{X\sim P_X} \left[\int_{\cY} \left|
 p_{\hat{G}_{\vtheta}}(y| X)-p_{Y|X}(y| X)\right| dy\right]   .
 \]
 We have
 \begin{eqnarray*}
 \Delta(P_{X,\hat{G}_{\vtheta}}, P_{X, Y})
 &=& \int_{\cX\times \cY} \left|p_{\hat{G}_{\vtheta}}(y| x)-p_{Y|X}(y| x)\right| p_X(x) dx dy \\
 &=& \int_{\cX\times \cY}  \left|p_{X,\hat{G}_{\vtheta}}(x,y)-p_{X,Y}(x,y)\right|  dx dy.
 \end{eqnarray*}
 It follows from Theorem \ref{consistency} that
 \[
 \Ebb_{\{X_i, Y_i\}_{i=1}^n} \Delta(P_{X,\hat{G}_{\vtheta}}(\eta, X), P_{X, Y})  \to 0.
 \]
 So Corollary 4.1 
 follows from the Markov inequality.

\end{proof}

Finally, we state the two lemmas used in the proofs above for ease of reference,
The first lemma is used in the proof of Lemma \ref{errsta}.
\begin{lemma}\label{2.9}
			If $\xi_{i}, i = 1,...m$ are $m$ finite linear combinations of Rademacher variables $\epsilon_j, j=1,..J$.
			Then
			\begin{equation}\label{rma}
			\mathbb{E}_{\epsilon_j,j=1,...J} \max _{1\leq i \leq m} |\xi_{i}| \leq C_6 (\log m)^{1/2} \max _{1\leq i \leq m}\left(\mathbb{E} \xi_{i}^{2}\right)^{1 / 2}
			\end{equation}
for some constant $C_6 > 0.$
		\end{lemma}
		\begin{proof}
			This result follows directly from Corollary 3.2.6 and  inequality (4.3.1)  in \cite{de2012decoupling} with $\Phi(x) = \exp(x^2).$
\end{proof}

The following approximation result about neural networks \citep{shen2019deep} is used in the proof of Lemma \ref{app1}.
\begin{lemma}\label{appshen}
 Let $f$ be a uniformly continuous function defined on $E \subseteq[-R, R]^{d}$. For arbitrary $L \in \mathbb{N}^{+}$ and $N \in \mathbb{N}^{+},$ there exists a function  ReLU network $f_{\phi}$ with width $3^{d+3} \max \left\{d\left\lfloor N^{1 / d}\right\rfloor, N+1\right\}$ and depth $12 L+14+2 d$ such that
$$
\|f-f_{\phi}\|_{L^{\infty}(E)} \leq 19 \sqrt{d} \omega_{f}^{E}\left(2 R N^{-2 / d} L^{-2 / d}\right),
$$
where, $\omega_{f}^{E}(t)$  is the  modulus of continuity of $f$ satisfying $\omega_{f}^{E}(t)\rightarrow 0$ as $t\rightarrow 0^+.$
\end{lemma}
\begin{proof}
This is  Theorem 4.3 in  \cite{shen2019deep}.
\end{proof}

\begin{singlespace}
\bibliography{gcds_bib}
\end{singlespace}

\end{document}